\documentclass[preprint]{elsarticle}
\usepackage{lineno,rotating,caption}
\usepackage{amsmath,amssymb,amsfonts,amsthm,anysize,setspace}
\usepackage{amssymb}
\usepackage{graphicx}
\usepackage[colorlinks]{hyperref}
\usepackage{caption}
\usepackage{subfigure}
\usepackage{amsthm}
\usepackage{natbib}
\usepackage{url}
\usepackage{pdftexcmds}
\usepackage{multirow}
\usepackage{enumerate}
\makeatletter
\newcommand{\myitem}[1]{%
\item[#1]\protected@edef\@currentlabel{#1}%
}
\makeatother
\makeatletter
\def\ps@pprintTitle{%
 \let\@oddhead\@empty
 \let\@evenhead\@empty
 \def\@oddfoot{}%
 \let\@evenfoot\@oddfoot}
\makeatother

\chardef\bslash=`\\ 

\newtheorem{theorem}{Theorem}[section]

\theoremstyle{definition}
\newtheorem{definition}{Definition}[section]
\newtheorem{remark}{Remark}

\newtheorem{example}{Example}

\journal{}
\begin{document}
\begin{frontmatter}

\title{A stability-enhanced nonstandard finite difference framework for solving one and two-dimensional nonlocal differential equations}
\author{Shweta Kumari$^{a}$}
\ead{shwetakri8062@gmail.com}
\author{Mani Mehra$^{a*}$}
\ead{mmehra@maths.iitd.ac.in}
\address{$^{a}$Department of Mathematics, Indian Institute of Technology Delhi\\
New Delhi-110016, India}
\cortext[cor1]{Corresponding author}

\begin{abstract}
Standard finite difference (SFD) schemes often suffer from limited stability regions, especially when applied in an explicit setup to partial differential equations. To address this challenge, this study investigates the efficacy of nonstandard finite difference (NSFD) schemes in enhancing the stability of explicit SFD schemes for 1D and 2D Caputo-type time-fractional diffusion equations (TFDEs). A nonstandard $L1$ approximation is proposed for the Caputo fractional derivative, and its local truncation error is derived analytically. This nonstandard $L1$ formulation is used to construct the NSFD scheme for a Caputo-type time-fractional initial value problem. The absolute stability of the resulting scheme is rigorously examined using the boundary locus method, and its performance is validated through numerical simulations on test examples for various choices of denominator functions. Based on this framework, two explicit NSFD schemes are developed for the 1D and 2D cases of the Caputo-type TFDE. Their stability is further assessed through the discrete energy method, with particular focus on the expansion of the stability region relative to SFD schemes. The convergence of the proposed NSFD schemes is also established. Finally, a comprehensive set of numerical experiments is conducted to demonstrate the accuracy and stability advantages of the proposed methods, with results presented through tables and graphical illustrations.
\end{abstract}

\end{frontmatter}
\numberwithin{equation}{section}

\section{Introduction}
Time-fractional differential equations have emerged as powerful mathematical models for capturing memory-dependent dynamics in complex systems \cite{podlubny1998fractional,samko1993fractional,oldham1974fractional}. These incorporate nonlocal temporal operators, allowing the current state of the system to depend on its entire evolution history, rather than just on the present configuration. This makes them particularly suitable for describing anomalous transport, hereditary materials, and relaxation processes \cite{kilbas2006theory,diethelm2010analysis,kumari2024numerical}. Among them, time-fractional diffusion equations (TFDEs) have garnered special attention due to their capacity to model subdiffusive behavior frequently observed in heterogeneous media such as porous rocks, biological tissues, and crowded cellular environments \cite{magin2010fractional,mehandiratta2020difference,mehandiratta2023well}. Applications span diverse domains, including viscoelasticity, heat and mass transport, electromagnetics, chaotic systems, and biological processes \cite{shah2019new,engheta2002fractional,kumari2025finite}. In particular, Caputo-type TFDEs are widely employed to model physical processes with temporal memory, such as turbulent flows \cite{carreras2001anomalous}, viscoelastic materials \cite{koeller1984applications}, and biological dynamics \cite{magin2010fractional}.\par
Despite their modeling advantages, TFDEs pose significant computational challenges due to the nonlocality of the fractional derivative. The Caputo derivative introduces a memory kernel that couples all previous time steps, leading to increased computational cost and storage demands. Hence, analytical solutions to such models are often difficult to obtain, and a broad spectrum of numerical methods has been developed to approximate their solutions accurately \cite{singh2025non}. These methods must contend with the reduced regularity of the solutions and the inherent nonlocality of the fractional operator \cite{jin2023numerical,kumari2024high}.
Moreover, SFD schemes, being one of the popular numerical methods for solving differential equations, often suffer from severe restrictions on the time step due to stability limitations, especially in an explicit setup \cite{kanwal2024explicit,chen2013superlinearly}. Explicit schemes are particularly attractive due to their simplicity and parallelizability, but their usefulness is often hindered by restrictive stability constraints. Additionally, high-order accuracy may require fine grids, making SFD schemes computationally expensive and less flexible for problems with complex geometry or nonlocal behaviour. To efficiently simulate long-time dynamics or large-scale systems modeled by TFDEs, it becomes imperative to design numerical schemes that offer improved stability properties. Therefore, the development of stability-enhanced numerical schemes is essential to unlock the potential of explicit formulations for solving time-fractional models.\par
NSFD methods, introduced by Mickens \cite{mickens1993nonstandard,mickens2002nonstandard}, provide an alternative approach for constructing discrete models of continuous systems. These schemes modify the standard discretization process, often through the use of nonlocal discretes for nonlinear terms or nonstandard denominator functions (DFs), to ensure consistency with the underlying dynamics.
One of the objectives of these constructions is to overcome the numerical instabilities that exist in the formulation of SFDs.  A key advantage of NSFD schemes is their ability to maintain essential structural properties of the model, such as positivity, boundedness, etc, independent of step sizes. Previous studies have demonstrated that NSFD schemes often yield better accuracy and qualitative behavior than standard methods \cite{mickens2004positivity}.
While the application of NSFD schemes to classical partial differential equations (PDEs) is well-established \cite{kayenat2024some,mickens2020note}, their use in fractional-order models 
remains relatively underdeveloped. Article \cite{moaddy2011non} proposing a nonstandard finite difference scheme for linear fractional PDEs in fluid mechanics with Riemann-Liouville derivative used nonstandard approximations for classical derivatives only.
Existing studies have proposed NSFD schemes for fractional PDEs involving Caputo derivative \cite{agarwal2018non,taghipour2022efficient,dwivedi2019numerical,anileyresults}, but they used nonstandard approximations only for the integer-order derivatives and rely on standard techniques such as the Crank–Nicolson method or spectral collocation for the Caputo fractional terms. 
Notably, a direct NSFD approximation of the Caputo fractional derivative itself is still lacking in the literature.\par
We fill this gap by considering the following initial-boundary value problem (IBVP) for one and two-dimensional Caputo-type TFDEs as
\begin{align}
\label{diff_eqn}
{}^{\text{C}}D_{0,t}^{\alpha}u(\mathbf{x},t) &= \mathcal{L}u(\mathbf{x},t) + f(\mathbf{x},t), \quad (\mathbf{x},t) \in \Omega \times (0,T], \\
\label{ini_eqn}
u(\mathbf{x},0) &= u^0(\mathbf{x}), \quad \mathbf{x} \in \Omega, \\
\label{bdry_eqn}
u(\mathbf{x},t) &= 0, \quad (\mathbf{x},t) \in \partial \Omega \times [0,T],
\end{align}
where $\Omega \subset \mathbb{R}^d$ is the spatial domain with $d = 1$ or $2$. The spatial differential operator $\mathcal{L}$ is defined as
\begin{align*}
\text{for } d=1: & \quad \mathcal{L}u(x,t) = \frac{\partial^2 u}{\partial x^2}(x,t), \quad \Omega = [0,L]; \\
\text{for } d=2: & \quad \mathcal{L}u(x,y,t) = \frac{\partial^2 u}{\partial x^2}(x,y,t) + \frac{\partial^2 u}{\partial y^2}(x,y,t), \quad \Omega = [0,L_x] \times [0,L_y].
\end{align*}
Here, ${}^{\text{C}}D_{0,t}^{\alpha}$ is the Caputo fractional operator and the equations \eqref{ini_eqn} and \eqref{bdry_eqn} are initial and Dirichlet boundary conditions. In this work, we construct a nonstandard $L1$ ($NSL1$) approximation for the Caputo fractional derivative of fractional order $\alpha \in (0,1)$ and validate it through numerical experiments on time-fractional initial value problems (IVPs). We also study the absolute stability of this $NSL1$ formulation via the boundary locus method and conclude it for the standard one. We then apply this approximation to develop fully explicit NSFD schemes for one and two-dimensional Caputo-type TFDEs \eqref{diff_eqn}. The explicit nature of these methods renders them computationally attractive for large-scale simulations and parallel implementations. The primary contribution of this study lies in the rigorous absolute stability analysis of the proposed $NSL1$ scheme, and the stability analysis of the explicit NSFD schemes aimed at mitigating the numerical instabilities typically associated with standard explicit methods.\par
The remainder of this article is organized as follows. Section \ref{sec2} introduces the necessary preliminaries and definitions required as the basis of this work. In Section \ref{sec3}, a nonstandard $L1$ approximation to the Caputo fractional derivative is proposed, and its local truncation error is derived. In Section \ref{sec4}, we analyze the absolute stability of the $NSL1$ scheme using the boundary locus method, supported by graphical illustrations. Using this result, we also conclude the absolute stability of the standard $L1$ scheme. In the same section, a set of numerical experiments is performed to confirm the accuracy and efficiency of the $NSL1$ method. Section \ref{sec5} presents the development of explicit NSFD schemes for the 1D and 2D TFDEs based on this $NSL1$ approximation, along with a thorough stability and convergence analysis. Numerical results are reported in Section \ref{sec6} to validate the theoretical findings. Finally, Section \ref{sec7} summarizes the main conclusions and outlines potential directions for future work.
\section{Preliminaries}\label{sec2}
This segment contains the fundamentals required for a basic understanding of the topic and building a foundation for the research carried out in this work.
\begin{definition}
\textbf{Caputo Derivative} \cite{li2015numerical}.\\
The left Caputo fractional derivative of order $\alpha>0$ of the function $y(t)$, $y\in C^m(a,b)$, is defined as
\begin{equation*}
_{C}D_{a,t}^{\alpha}y(t) = \frac{1}{\Gamma(m-\alpha)}\int_{a}^{t}(t-\zeta)^{m-\alpha-1}y^{(m)}(\zeta)d\zeta,
\end{equation*}
where $m-1<\alpha \leq m$, $m \in \mathbb{N}$. The Caputo derivative for a particular value of $m=1$ so that $0<\alpha<1$ is given as
\begin{equation}\label{caputo}
_{C}D_{a,t}^{\alpha}y(t) = \frac{1}{\Gamma(1-\alpha)}\int_{a}^{t}(t-\zeta)^{-\alpha}y^\prime(\zeta)d\zeta.
\end{equation}
\end{definition}
\begin{definition}\textbf{Linear Interpolation} \cite{qiao2022fast}.\\
The linear Lagrange interpolating polynomial of a function $y(t)$ on some interval $[t_{k-1},t_k],\ k\geq 1,$ is defined as
\begin{equation}\label{lin_int}
I_sy(t) = \frac{t-t_k}{t_{k-1}-t_k}y(t_{k-1})+\frac{t-t_{k-1}}{t_k-t_{k-1}}y(t_k).
\end{equation}
Thus, we have the interpolation error as
\begin{equation}\label{int_error}
T_s(t)=y(t)-I_{s}y(t)=y''(\xi_k)(t-t_{k-1})(t-t_{k}),\quad \xi_k \in (t_{k-1},t_k).
\end{equation}
\end{definition}
\begin{definition} \textbf{$L1$ method} \cite{lin2007finite}.\\
The $L1$ method is the most popular approximation for the Caputo fractional derivative, defined as
\begin{align}\label{L1}
_{C}D_{0,t_j}^{\alpha}y(t)=&\frac{1}{\Gamma(1-\alpha)}\sum_{k=0}^{j-1}\frac{y(t_{k+1})-y(t_k)}{t_{k+1}-t_k}\int_{t_k}^{t_{k+1}}(t_j-\zeta)^{-\alpha}d\zeta+\varepsilon^j,
\end{align}
where the approximation error $\varepsilon^j=\mathcal{O}(\tau^{2-\alpha})$, $\tau$ is the step size, provided $y(t)\in C^2[0,T]$.
\end{definition}
\begin{definition}\textbf{Mittag-Leffler function} \cite{kilbas2006theory}.\\
The following defines a Mittag-Leffler function as
\begin{equation}\label{mittag}
    E_{\alpha,\beta}(z)=\sum_{j=0}^\infty \frac{z^j}{\Gamma(\alpha j+\beta)},\ z\in\mathbb{C},
\end{equation}
where $\alpha$ is positive and $\beta\in\mathbb{R}$ are arbitrary constants.
\end{definition}
\subsection{\textbf{Rules for constructing  NSFD approximation}}
The NSFD schemes are discrete models to numerically solve differential equations that model important phenomena of natural and engineering sciences. An NSFD scheme is constructed using many rules stated by Mickens \cite{mickens1993nonstandard,mickens2002nonstandard,mickens2000applications}. Some of the requirements for our NSFD scheme construction are as follows:
\begin{enumerate}
    \item The orders of the discrete derivatives should be equal to the orders of the corresponding derivatives of the differential equations.
    \item DFs for the discrete derivatives must generally be expressed in terms of specific structured functions of the step-sizes.
\end{enumerate}
 In general, the discrete first derivative takes the form
 \[ \frac{du}{dt}\approx \frac{u_{k+1}-\nu(\tau)u_k}{\varphi(\tau)},
 \]
 where $\nu(\tau)$ and $\varphi(\tau)$ depend on the step-size $\tau$, and satisfy the conditions
 \[
 \nu(\tau)=1+\mathcal{O}(\tau^2)\quad \& \quad \varphi(\tau)=\tau+\mathcal{O}(\tau^2).
 \]
 Also, the functions $\nu(\tau)$ and $\varphi(\tau)$ may also depend explicitly on the various parameters that appear in the differential equations. Moreover, the function $\nu(\tau)$ is a continuous function of a step size $h$ and satisfies condition $0<\nu(\tau)<1,\ \tau\rightarrow 0$. In our work, $\nu(\tau)=1$. Some examples of the functions $\varphi(\tau)$ that satisfy these conditions are \cite{mickens2005advances,mickens2007calculation}
 \[
 \varphi(\tau)=\tau,\ \sinh(\tau),\ e^{\tau}-1,\ \sin(\tau),\ \frac{1-e^{(-\lambda \tau)}}{\lambda}, \text{ etc,.}
 \]
 There is no basis for the best choice of the function $\phi(h)$. For more details on the non-standard finite difference method, see \cite{mickens1993nonstandard,mickens2002nonstandard,mickens2000applications}.
 The NSFD approximation to be used for the spatial derivative in \eqref{diff_eqn} is as follows:
\begin{equation}\label{nsspace}
{\partial_{x}^{2}}u(x_m,t_n) = \frac{u(x_{m+1},t_n)-2u(x_m,t_n)+u(x_{m-1},t_n)}{\psi(h)}+\mathcal{O}(h^2),
\end{equation}
where, $\psi(h)=h^2+\mathcal{O}(h^3)$. 
\section{\texorpdfstring{$NSL1$}{NSL1} Approximation to Caputo Fractional Derivative}\label{sec3}
In this section, a nonstandard $L1$ approximation to the Caputo fractional derivative \eqref{caputo} is constructed. To begin with, the time interval $[0,T]$ is discretized as $\mathcal{T}=\{t_n:t_n=n\tau,\ n=0,1,2,\dots,N,\ N\in\mathbb{Z}^+\}$, by dividing it into $N$ equal subintervals of step length $\tau=\frac{T}{N}$. We use the method of Lagrange interpolation to approximate the unknown in the integrand of the Caputo fractional derivative \eqref{caputo}. The standard linear Lagrange interpolating polynomial of a function $y(t),\ t\in[t_{n-1},t_n],\ n\geq 1,$ is given by \eqref{lin_int}. Hence, employing the rules of NSFD approximations by Mickens \cite{mickens1993nonstandard}, the nonstandard version of \eqref{lin_int} for $y(t)$ appears as
\begin{equation}\label{ns_linint}
I_{ns}y(t) = \frac{t-t_n}{-\varphi(\tau)}y(t_{n-1})+\frac{t-t_{n-1}}{\varphi(\tau)}y(t_n),
\end{equation}
where, $\varphi(\tau)=\tau+\mathcal{O}(\tau^2)$. Then, the $NSL1$ approximation to the Caputo fractional derivative \eqref{caputo} using \eqref{ns_linint} can be written as
\begin{align*}
    _{C}D_{0,t}^{\alpha}y(t_n) &= \frac{1}{\Gamma(1-\alpha)}\int_{0}^{t_n}(t_n-\zeta)^{-\alpha}\frac{dy}{dt}(\zeta)d\zeta,\\
    &\approx \frac{1}{\Gamma(1-\alpha)}\sum_{k=0}^{n-1}\int_{t_k}^{t_{k+1}}(t_n-\zeta)^{-\alpha}\frac{dI_{ns}y}{dt}(\zeta)d\zeta,\\
    & =\frac{1}{\Gamma(1-\alpha)}\sum_{k=0}^{n-1}\frac{y(t_{k+1})-y(t_k)}{\varphi(\tau)}\int_{t_k}^{t_{k+1}}(t_n-\zeta)^{-\alpha}d\zeta,\\
    & = \frac{\tau^{1-\alpha}}{\varphi(\tau)\Gamma(2-\alpha)}\left[b_{1,\alpha}y(t_n)+\sum_{k=1}^{n-1}(b_{k+1,\alpha}-b_{k,\alpha})y(t_{n-k})-b_{n,\alpha}y(t_0)\right],\\
    & =:\ _{C}D_{N_{ns}}^{\alpha}y(t_n),
\end{align*}
where $b_{k,\alpha}=k^{1-\alpha}-(k-1)^{1-\alpha},\ k=1,2,\dots,n$. In the standard case, $\varphi(\tau)=\tau$. Moreover,
\begin{align*}
    \frac{\tau^{1-\alpha}}{\varphi(\tau)}&=\tau^{1-\alpha}\varphi(\tau)^{-1},\\
    &=\tau^{1-\alpha}\left(\tau+\mathcal{O}(\tau^2)\right)^{-1},\\
    &=\tau^{-\alpha}\left(1+\mathcal{O}(\tau)\right)^{-1}.
\end{align*}
Using binomial expansion of $(1+\mathcal{O}(\tau))^{-1}$ and $\varphi(\tau)^{-\alpha}$, it becomes
\begin{align*}
    \frac{\tau^{1-\alpha}}{\varphi(\tau)}&=\tau^{-\alpha}\left(1-\mathcal{O}(\tau)\right),\\
    &=\left(\tau^{-\alpha}+\mathcal{O}(\tau^{1-\alpha})\right),\\
    &=\varphi(\tau)^{-\alpha}=\phi(\tau)^{-1},
\end{align*}
where, $\phi(\tau)=\tau^{\alpha}+\mathcal{O}(\tau^{1+\alpha})$. Hence,
\begin{equation}\label{nscaputo}
   _{C}D_{N_{ns}}^{\alpha}y(t_n)= \frac{\phi(\tau)^{-1}}{\Gamma(2-\alpha)}\left[b_{1,\alpha}y(t_n)+\sum_{k=1}^{n-1}(b_{k+1,\alpha}-b_{k,\alpha})y(t_{n-k})-b_{n,\alpha}y(t_0)\right].
\end{equation}
Next, the following theorem presents the local truncation error estimate of the NSFD $L1$ approximation \eqref{nscaputo} to the Caputo fractional derivative \eqref{caputo}.
\begin{theorem}\label{trunclemma}
Let $y(t)\in C^2[0,T]$, then there exists a constant K such that $\forall\ t_n\in \mathcal{T},$ we have
\begin{equation}\label{trunc_eqn}
\left|_{C}D_{0,t}^{\alpha}y(t_n)-\ _{C}D_{N_{ns}}^{\alpha}y(t_n)\right|\leq K\tau^{2-\alpha},\quad n=1,2,\dots,N,
\end{equation}
$K$ being a finite positive constant independent of $\tau$.
\end{theorem}
\begin{proof}
    Let the local truncation error of the nonstandard linear Lagrange interpolating polynomial \eqref{ns_linint}, $t\in[t_{n-1},t_n],\ n\geq1$, be denoted and defined as
\begin{align*}
    T_{ns}(t)&=y(t)-I_{ns}y(t),\quad \\
    &=y(t)-\left[\frac{t-t_n}{-\varphi(\tau)}y(t_{n-1})+\frac{t-t_{n-1}}{\varphi(\tau)}y(t_n)\right],\\
    &=y(t)-\frac{\tau}{\varphi(\tau)}\left[\frac{t-t_n}{-\tau}y(t_{n-1})+\frac{t-t_{n-1}}{\tau}y(t_n)\right],\\
    &=y(t_{n-1}+\epsilon)-\frac{\tau}{\varphi(\tau)}\left[\frac{t_{n-1}+\epsilon-t_n}{-\tau}y(t_{n-1})+\frac{t_{n-1}+\epsilon-t_{n-1}}{\tau}y(t_{n-1}+\tau)\right],\\
    &=\left[y(t_{n-1})+\epsilon y'(t_{n-1})+\frac{\epsilon^2}{2}y''(t_{n-1})+\cdots\right]-\frac{\tau}{\varphi(\tau)}\left[\frac{\tau-\epsilon}{\tau}y(t_{n-1})\right]\\
    &\hspace{3.6cm}-\frac{\epsilon}{\varphi(\tau)}\left[y(t_{n-1})+\tau y'(t_{n-1})+\frac{\tau^2}{2}y''(t_{n-1})+\cdots\right],\\
    &=\left[1-\frac{\tau}{\varphi(\tau)}\right]y(t_{n-1})+\epsilon\left[1-\frac{\tau}{\varphi(\tau)}\right]y'(t_{n-1})+\frac{\epsilon}{2}\left[\epsilon-\frac{\tau^2}{\varphi(\tau)}\right]y''(t_{n-1})+\cdots,\\
    &=\left[1-\frac{\tau}{\varphi(\tau)}\right]\left(y(t_{n-1})+\epsilon y'(t_{n-1})\right)+\frac{\epsilon}{2}\left[\epsilon-\frac{\tau^2}{\varphi(\tau)}\right]y''(t_{n-1})+\cdots.
\end{align*}
Now, since $\varphi(\tau)=\tau+\mathcal{O}(\tau^2)$, thus, as $\tau\rightarrow 0$, $\ \varphi(\tau)\approx\tau$. Therefore, 
\begin{align*}
    T_{ns}(t)&=\frac{(t-t_{n-1})}{2}\left[(t-t_{n-1})-(t_n-t_{n-1})\right]y''(t_{n-1})+\cdots,\\
    &=\frac{(t-t_{n-1})(t-t_{n})}{2}y''(\xi_n),\quad \xi_n\in(t_{n-1},t_n), \quad \text{as }\tau\rightarrow 0.
\end{align*}
Further proceeding as in \cite[Lemma 3.1]{lin2007finite}, inequality \eqref{trunc_eqn} can be easily obtained.
\end{proof}
\section{\textbf{NSFD Scheme for Caputo-type Time-fractional IVP}
}\label{sec4}
Following the Caputo-type fractional IVP is considered for further computations in this section.
\begin{equation}\label{Caputo_eqn}
_CD^\alpha_{0,t}y(t)=f(t),\quad y(0)=y_0,\quad t\in[0,T].
\end{equation}
The discretization of the time interval $[0, T]$ is performed in the same way as in the previous section. Considering the above equation at time level $t_n$ and using $y(t_n)\approx y^n$ with $NSL1$ approximation \eqref{nscaputo} for the Caputo fraction derivative yields
\begin{align}\notag
&\hspace{1.5cm}\frac{1}{\Gamma(1-\alpha)}\sum_{k=0}^{n-1}\frac{y^{k+1}-y^k}{\tau}\int_{t_k}^{t_{k+1}}(t_n-\zeta)^{-\alpha}d\zeta=f^n,\\\notag
&\hspace{0.6cm}\frac{\phi(\tau)^{-1}}{\Gamma(2-\alpha)}\left[b_{1,\alpha}y^n+\sum_{j=1}^{n-1}\left(b_{j+1,\alpha}-b_{j,\alpha}\right)y^{n-j}-b_{n,\alpha}y^0\right]=f^n,\\\label{ivpscheme}
&\hspace{0.9cm}y^n=b_{n,\alpha}y^0+\sum_{j=1}^{n-1}\left(b_{j,\alpha}-b_{j+1,\alpha}\right)y^{n-j}+\phi(\tau)\Gamma(2-\alpha)f^n.
\end{align}
\subsection{\textbf{Absolute stability analysis for} \texorpdfstring{$\boldsymbol{NSL1}$}{NSL1} \textbf{method}}
This section analyses the absolute stability of the $NSL1$ Method \eqref{nscaputo}. The model test problem for the Caputo fractional derivative \eqref{caputo} with scalar parameter $\lambda$ is written as
\begin{equation*}
_CD^\alpha_{0,t}y(t)=\lambda y,\quad y(0)=y_0,\quad t\in[0,T].
\end{equation*}
The solution of the above IVP is given by \cite{kilbas2006theory}
\begin{equation*}
    y(t)=y_0 E_{\alpha,1}(\lambda t^\alpha),
\end{equation*}
where $E_{\alpha,1}$ is the Mittag-Leffler function \eqref{mittag}. In the above solution, $\lim_{t\rightarrow\infty}y(t)=0$ and hence, bounded for $|arg(\lambda)|>\alpha\pi/2$ \cite{diethelm2010analysis}. Thus, absolute stability requires the approximate numerical solution of the above IVP to imitate the same behavior, i.e., $y^n\rightarrow 0$ as $t_n\rightarrow\infty$ under the same restrictions. Hence, replacing $f^n$ by $\lambda y^n$ in \eqref{ivpscheme}, we get
\begin{equation*}
\left(\frac{b_{1,\alpha}}{\Gamma(2-\alpha)}-\lambda\phi(\tau)\right)y^n=\frac{1}{\Gamma(2-\alpha)}\left[b_{n,\alpha}y^0+\sum_{j=1}^{n-1}\left(b_{j,\alpha}-b_{j+1,\alpha}\right)y^{n-j}\right].
\end{equation*}
The above homogeneous linear diﬀerence equation with constant coeﬃcients has solutions of the form $y^n = ar^n$ \cite{elaydiintroduction}. Hence, considering $\hat{\tau}=\lambda\tau^{\alpha}$, $d_{k,\alpha}=\frac{b_{k,\alpha}}{\Gamma(2-\alpha)},\ k=1,2,\dots,n$, and substituting the value of $y^n$ in the above equation provides
\begin{equation}\label{stab_pol}
   \left(d_{1,\alpha}-\hat{\tau}\left(1+\mathcal{O}(\tau)\right)\right)r^n+\sum_{j=1}^{n-1}\left(d_{j+1,\alpha}-d_{j,\alpha}\right)r^{n-j}-d_{n,\alpha}r^0=0. 
\end{equation}
The left-hand side of the above equation is the stability polynomial $p(r)$ of degree $n$. Next, the values of $\hat{\tau}$ are determined for which the roots of $p(r)$ satisfy the strict root condition, i.e., $|r|<1$. We employ the boundary locus method here for further analysis. Substituting $r=e^{is}$ and letting $\tau\rightarrow0$ in the above equation gives
\begin{align*}
&(d_{1,\alpha}-\hat{\tau})e^{ins}+\sum_{j=1}^{n-1}(d_{j+1,\alpha}-d_{j,\alpha})e^{i(n-j)s}-d_{n,\alpha}e^{i0s}=0,
    \end{align*}
    \begin{align*}
    &\hspace{0.6cm}(d_{1,\alpha}-\hat{\tau})(\cos ns+i\sin ns)+\sum_{j=1}^{n-1}(d_{j+1,\alpha}-d_{j,\alpha})(\cos (n-j)s+i\sin (n-j)s)-d_{n,\alpha}=0,\\
    &d_{1,\alpha}\cos ns+id_{1,\alpha}\sin ns-d_{n,\alpha}+\sum_{j=1}^{n-1}\cos (n-j)s(d_{j+1,\alpha}-d_{j,\alpha})+i\sum_{j=1}^{n-1}\sin (n-j)s(d_{j+1,\alpha}-d_{j,\alpha})\\
    &\hspace{6cm}=\hat{\tau}(\cos ns+i\sin ns),
\end{align*}
which implies
\begin{align*}
   \hat{\tau}(s)&= \frac{d_{1,\alpha}\cos ns-d_{n,\alpha}+\sum_{j=1}^{n-1}\cos (n-j)s(d_{j+1,\alpha}-d_{j,\alpha})+i\left(d_{1,\alpha}\sin ns+\sum_{j=1}^{n-1}\sin (n-j)s(d_{j+1,\alpha}-d_{j,\alpha}\right)}{(\cos ns+i\sin ns)},\\
&= d_{1,\alpha}\cos^2 ns-d_{n,\alpha}\cos ns+\sum_{j=1}^{n-1}\cos (n-j)s\cos ns(d_{j+1,\alpha}-d_{j,\alpha})+d_{1,\alpha}\sin^2ns\\
   &\hspace{.3cm}+\sum_{j=1}^{n-1}\sin (n-j)s\sin ns(d_{j+1,\alpha}-d_{j,\alpha})+i\bigg(d_{1,\alpha}\sin ns\cos ns+\sum_{j=1}^{n-1}\sin (n-j)s\cos ns(d_{j+1,\alpha}-d_{j,\alpha})\\
&\hspace{.3cm}+d_{1,\alpha}\cos ns\sin ns-d_{n,\alpha}\sin ns+\sum_{j=1}^{n-1}\sin (n-j)s\cos (n-j)s(d_{j+1,\alpha}-d_{j,\alpha})\bigg),\\
&= d_{1,\alpha}-d_{n,\alpha}\cos ns+\sum_{j=1}^{n-1}\cos ns(d_{j+1,\alpha}-d_{j,\alpha})\\
&\hspace{.3cm}+i\bigg(d_{1,\alpha}\sin 2ns-d_{n,\alpha}\sin ns+\sum_{j=1}^{n-1}\sin (2n-j)s(d_{j+1,\alpha}-d_{j,\alpha})\bigg),
\end{align*}
\begin{align*}
\text{where, }\ \hat{x}(s)&=d_{1,\alpha}-d_{n,\alpha}\cos ns+\sum_{j=1}^{n-1}\cos ns(d_{j+1,\alpha}-d_{j,\alpha}).\\
    \hat{y}(s)&=d_{1,\alpha}\sin 2ns-d_{n,\alpha}\sin ns+\sum_{j=1}^{n-1}\sin (2n-j)s(d_{j+1,\alpha}-d_{j,\alpha}).
\end{align*}
The plot of the locus of points $\hat{x}(s)$ and $\hat{y}(s)$, $0\leq s < 2\pi$, for different values of $\alpha$ is illustrated in Figure \ref{absolute}. 
\begin{figure}[ht]
    \begin{center}
	\centering
		\subfigure[$\alpha=0.2$]{
	\includegraphics[width=7.714cm]{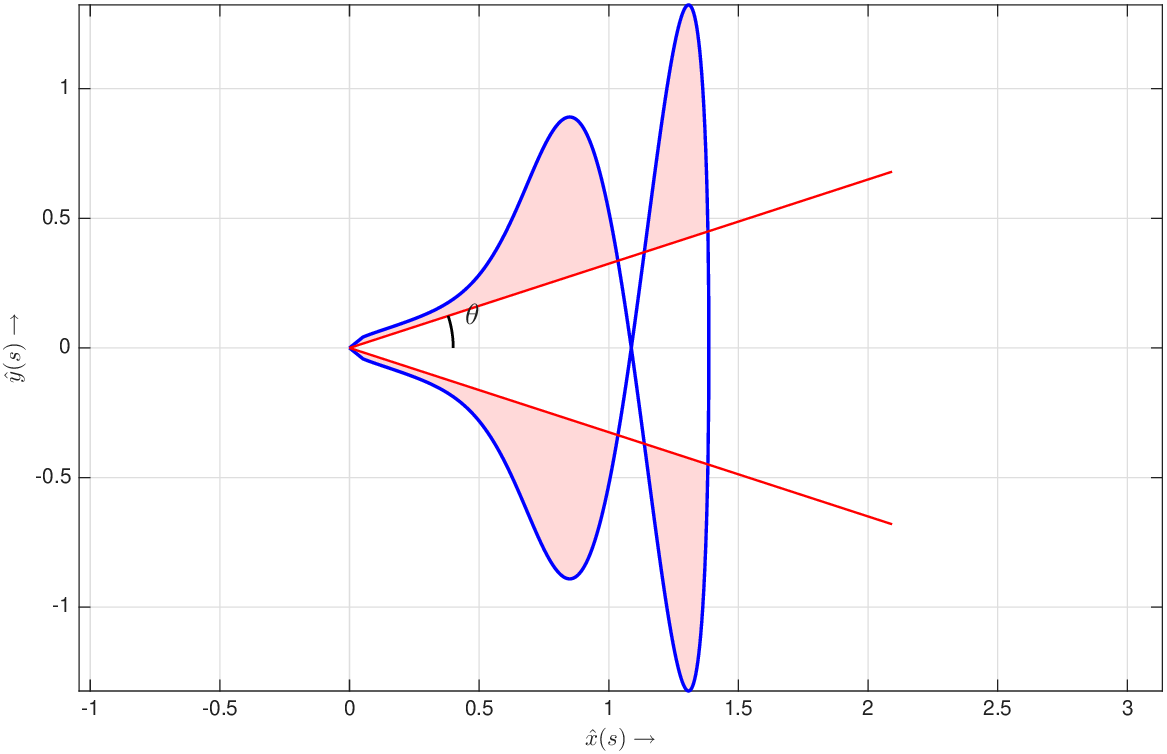}
			\label{subfigure1}} 
		\subfigure[$\alpha=0.4$]{
	\includegraphics[width=7.714cm]{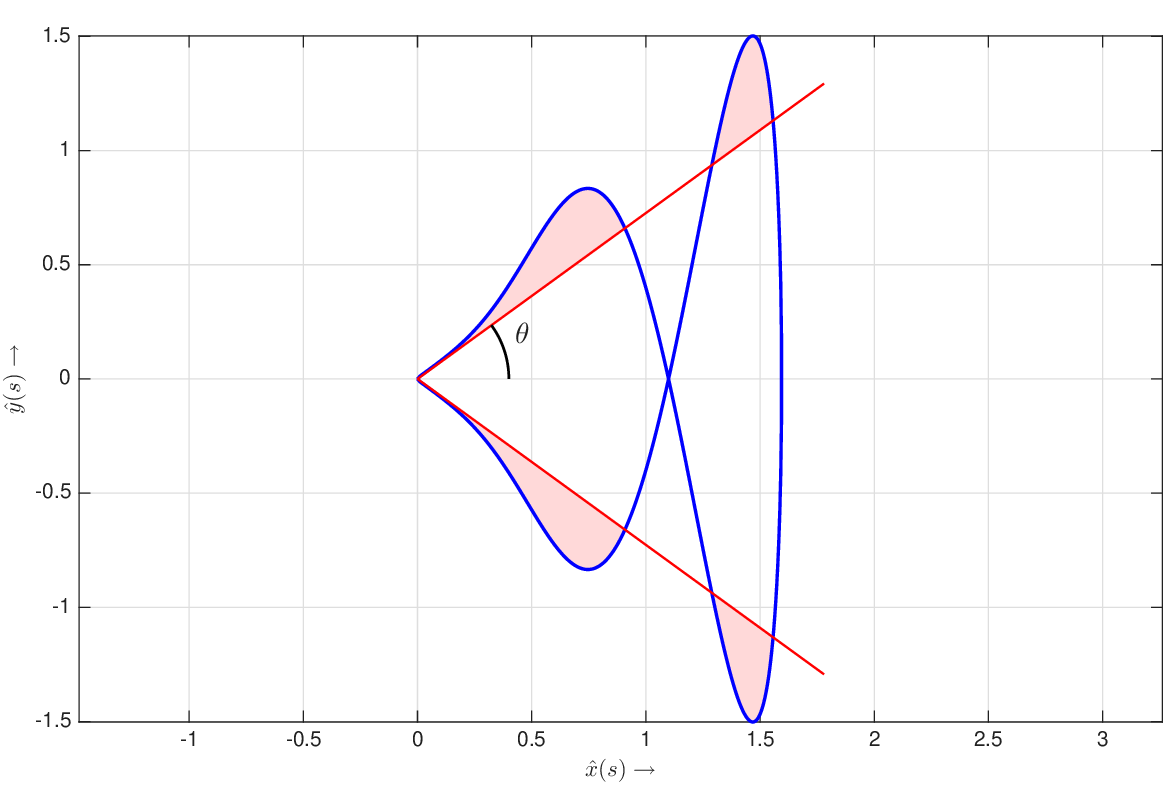}			 \label{subfigure2}} \\
    \subfigure[$\alpha=0.6$]{
	\includegraphics[width=7.714cm]{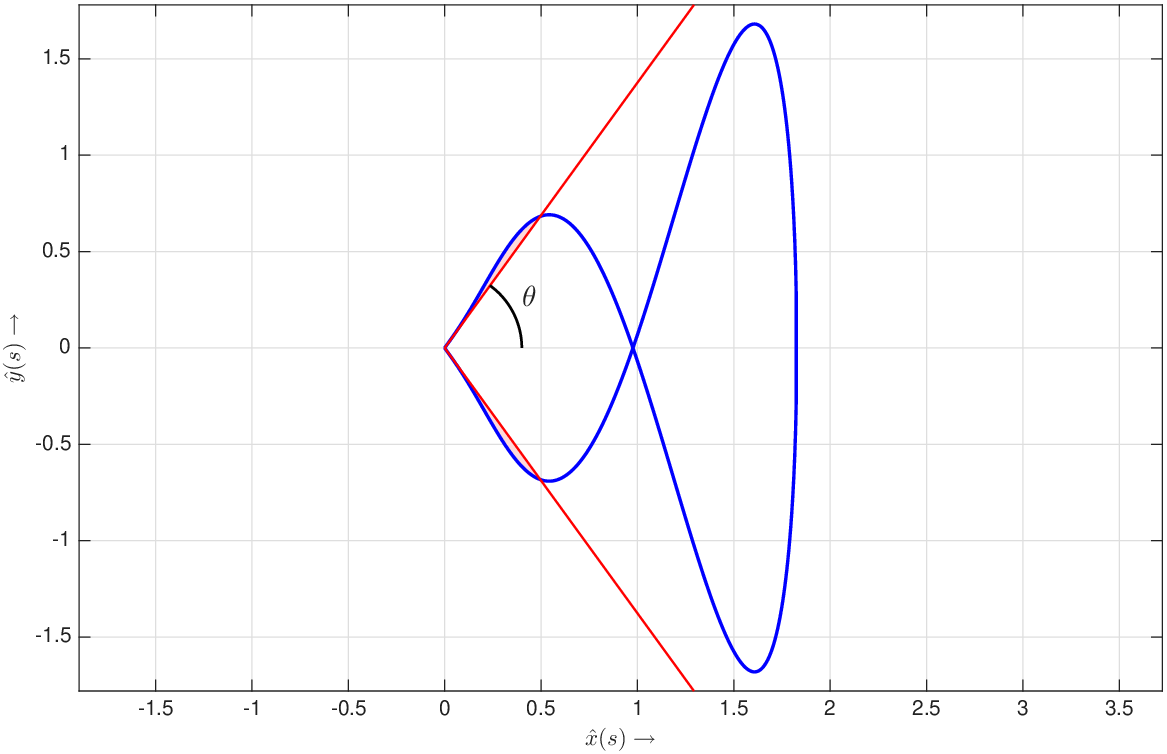}
			\label{subfigure3}} 
		\subfigure[$\alpha=0.8$]{
	\includegraphics[width=7.714cm]{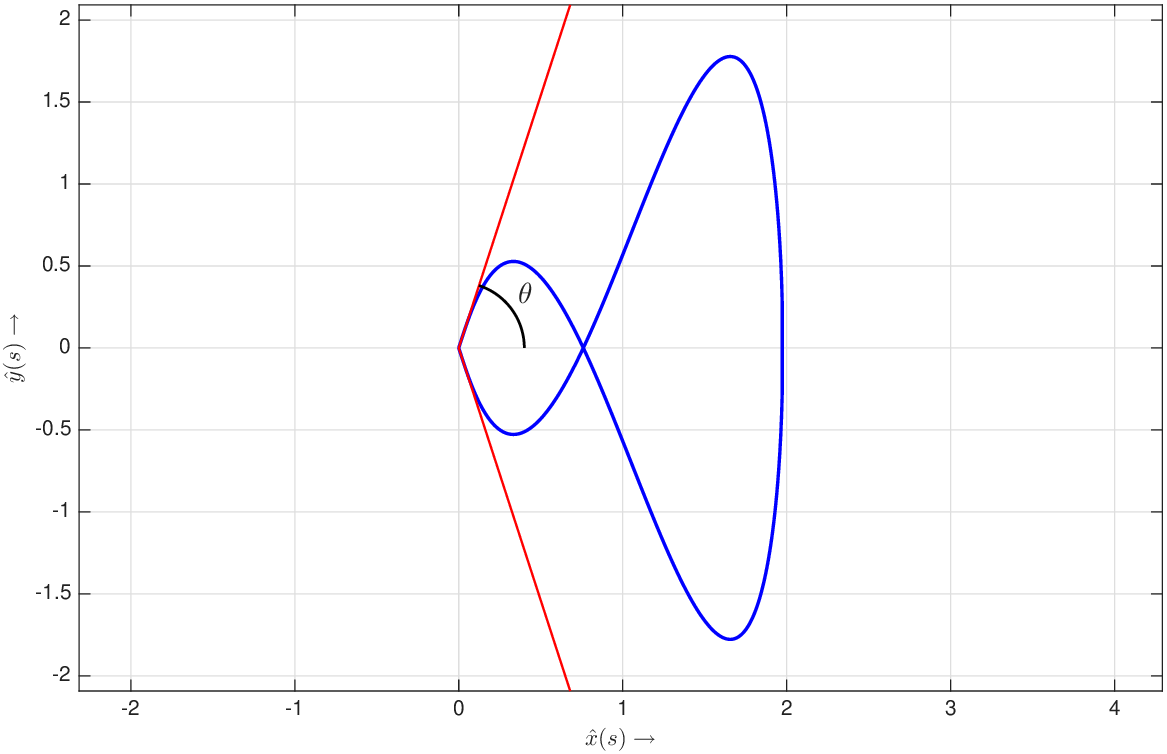}			 \label{subfigure4}}
    \caption{Locus plots of $\hat{x}(s)$ and $\hat{y}(s)$ for $0\leq s < 2\pi$ at different values of $\alpha$ with $n=100000$ and $\theta=\frac{\alpha\pi}{2}$.} 
    \label{absolute}
    \end{center}
\end{figure}
Here, the curve in all subfigures divides the plane into two closed and one open subregions, and each subregion is checked for the region of absolute stability. For this, we choose the test points as $\hat{x}(s)=-0.5,0.5$, and 1.5, along with $\hat{y}(s)=0$ and check whether the roots $r^n$ of the stability polynomial $p(r)$ \eqref{stab_pol} at these points satisfy the strict root condition $|r|<1$. 
We consider the root vector $r=(r_1,r_2,\dots,r_n)$ at each $n$ and plot the maximum absolute value of $r$ i.e., $r_{max}=\max_{1\leq j\leq n} |r_j|$ against $n$ for fixed value of $N=1000$ and different values of $\alpha$ for $\hat{x}(s)=-0.5,\ 0.5$, and 1.5, $\hat{y}(s)=0$ in Figure \ref{figure:2}. 
\begin{figure}[!h]
    \begin{center}
	\centering
		\subfigure[$\alpha=0.2$]{
	\includegraphics[width=7.714cm]{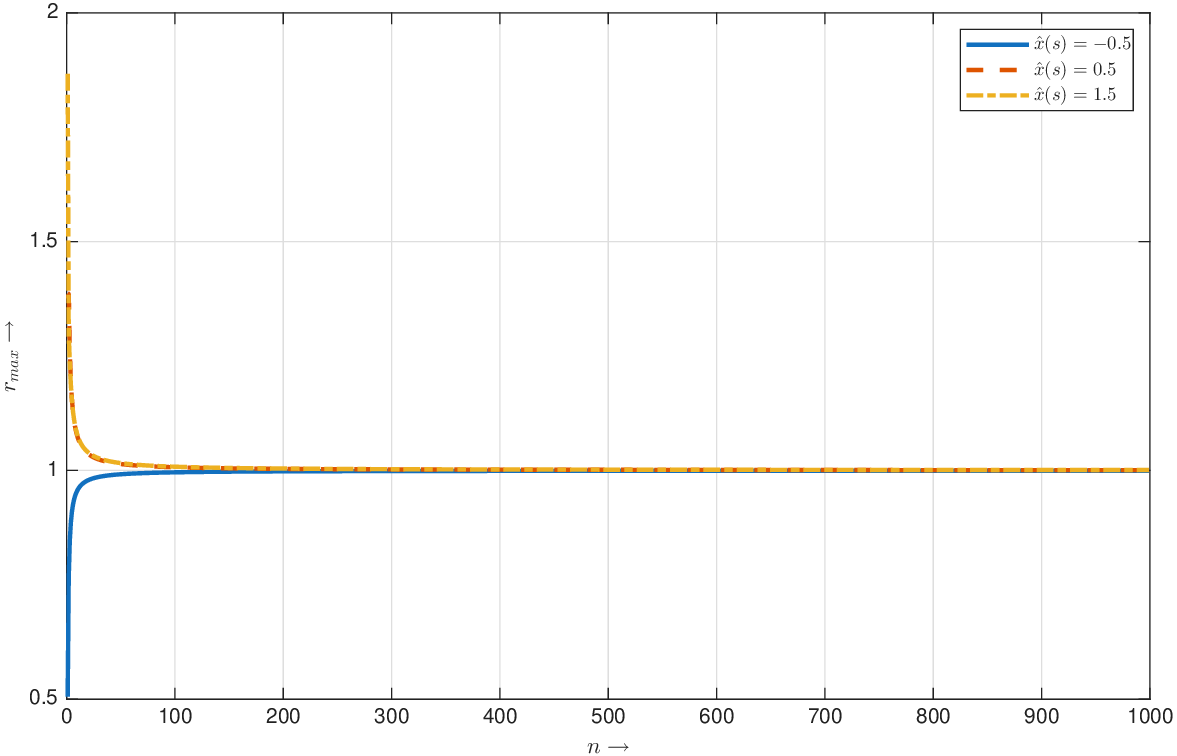}
		\label{subfigure5}} 
		\subfigure[$\alpha=0.4$]{
	\includegraphics[width=7.714cm]{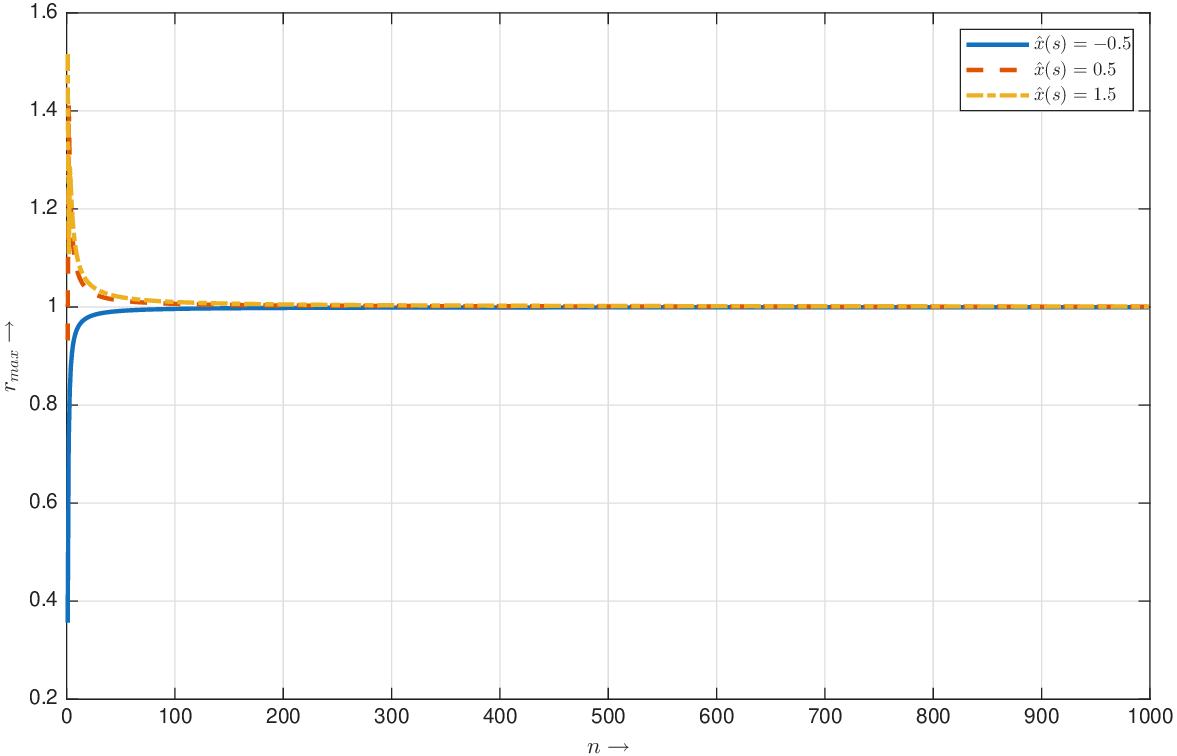}	
    \label{subfigure6}}
    \\
    \subfigure[$\alpha=0.6$]{
	\includegraphics[width=7.714cm]{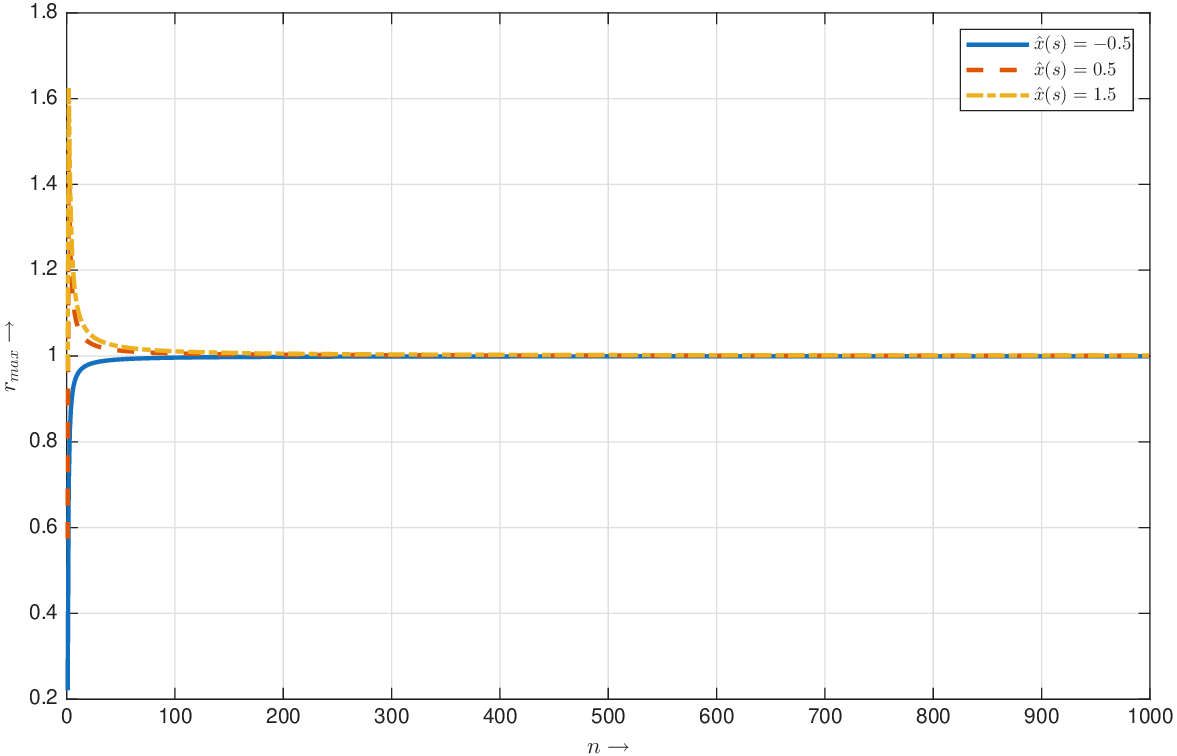}
			\label{subfigure51}} 
		\subfigure[$\alpha=0.8$]{
	\includegraphics[width=7.714cm]{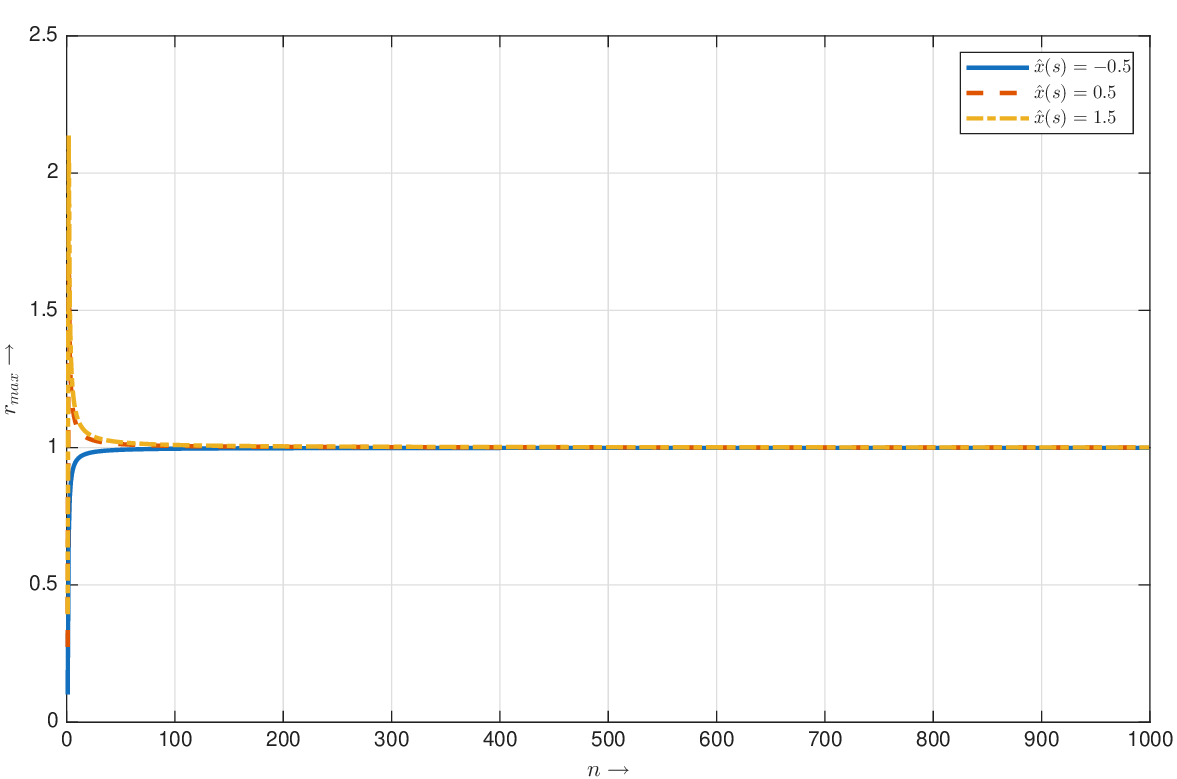}	
    \label{subfigure61}}
    \caption{Curve plots of $r_{max}$ against $n$ with $N=100$ at $\alpha=0.5$ and different values of $\hat{x}(s)$.} \label{figure:2}
    \end{center}
\end{figure}
It can be seen from these subfigures that $r_{max}\geq1$ for $\hat{x}(s)=0.5$ and 1.5, but  <1 for $\hat{x}(s)=-0.5$. Hence, it is concluded that although the method is more stable for higher values of $\alpha$, their absolute stability region includes the entire left half of the plane, and thus, the $NSL1$ method \eqref{nscaputo} for Caputo fractional derivative is $A(\alpha)$-stable.
\begin{remark}
    Since, in the standard case, $\phi(\tau)=\tau^\alpha$. Hence, from the above computations, it can be concluded that the standard $L1$ method \eqref{L1} for the Caputo derivative approximation is also $A(\alpha)$-stable.
\end{remark}
\subsection{\textbf{Numerical experiments for}\texorpdfstring{$\boldsymbol{NSL1}$}{NSL1} \textbf{method}}
The accuracy and efficiency of the $NSL1$ method \eqref{nscaputo} using scheme \eqref{ivpscheme} are inspected by conducting several numerical experiments over test examples. Here, two test examples are considered, first with the state function being $\alpha$ dependent and having a zero initial value, and second being $\alpha$ independent and having a non-zero initial value. The experiments are conducted for different values of $N$ at different $\alpha$s to conclude the convergence of these NSFD approximations. Various options of DFs are considered, as there is no fixed rule for their best choice. The exact solution for each test example is known in advance, and thus facilitates the comparison of errors. Furthermore, the $L_\infty$ norm is used to compute errors between exact and approximate solutions at the final time $T$, along with the rate of convergence, defined as
\begin{equation}\label{error}
E_{\infty}(N)=\left|u(t_N)-u^N\right|\quad \text{ and }\quad {rate} =\log_2\left(\frac{E_\infty(N/2)}{E_\infty(N)}\right).
\end{equation}
\begin{example}\label{ex1}
Consider the Caputo-type IVP \eqref{Caputo_eqn} with following data:
\begin{align*}
    y_0=0 \text{ and }f(t)=\frac{t^2\Gamma\left(3+\alpha\right)}{2},\quad t\in[0,1].
\end{align*}
A clear observation for the exact solution to the above IVP is $y(t)=t^{2+\alpha}$.
\end{example}
\begin{example}\label{ex2}
Consider the Caputo-type IVP \eqref{Caputo_eqn} with following data:
\begin{align*}
    y_0=1 \text{ and }f(t)=\frac{2t^{2-\alpha}}{\Gamma\left(3-\alpha\right)},\quad t\in[0,1].
\end{align*}
The exact solution to the above IVP is given by $y(t)=t^2+1$.
\end{example}
\begin{table}[ht]
\caption{$L_\infty$ errors and corresponding convergence orders for Example \ref{ex1} at $T=1$ for different $\alpha$, $N$ and $\varphi(\tau)$.}
    \label{table:1}
    \centering
    \begin{tabular}{c c c c c c c c c c}
    \hline
\multirow{2}{1em}{$\alpha$} & \multirow{2}{1em}{$N$} &\multicolumn{2}{c} 
   {$\varphi(\tau)=\tau$} &\multicolumn{2}{c} 
   {$\varphi(\tau)=1000(e^{\frac{\tau}{1000}}-1)$} &\multicolumn{2}{c} 
   {$\varphi(\tau)=\sin{\tau}$} &\multicolumn{2}{c} 
   {$\varphi(\tau)=\sinh{\tau}$}\\
   \cline{3-10} 
  & & $E_\infty$ & {rate} & $E_\infty$ & {rate} & $E_\infty$ & {rate} & $E_\infty$ & {rate}\\ 
  \hline
$0.3$& $10$ & 7.1000e-03 &  & 7.2000e-03 &  & 5.4000e-03 &  & 8.8000e-03 & \\ 
  & $20$ & 2.3000e-03 & 1.6262 & 2.4000e-03 & 1.5981 & 1.9000e-03 & 1.5070 & 2.8000e-03 & 1.6521\\ 
  & $40$ & 7.5669e-04 & 1.6039 & 7.6920e-04 & 1.6191 & 6.5244e-04 & 1.5421& 8.6094e-04 & 1.7014\\ 
  & $80$ & 2.4208e-04 & 1.6442 & 2.4833e-04 & 1.6311 & 2.1603e-04 & 1.5946 & 2.6813e-04 & 1.6830\\  
  & $160$ & 7.6786e-05 & 1.6566 & 7.9911e-05 & 1.6358 & 7.0275e-05 & 1.6201 & 8.3297e-05 & 1.6866\\ 
  & $320$ & 2.4200e-05 & 1.6658 & 2.5763e-05 & 1.6331 & 2.2573e-05 & 1.6384 & 2.5828e-05 & 1.6893\\ 
  \hline
$0.5$& $10$ &  2.1300e-02 &  & 2.1400e-02 &  & 1.9600e-02 &  & 2.3000e-02 & \\ 
  & $20$ & 7.9000e-03 & 1.4309 & 7.9000e-03 & 1.4328 & 7.5000e-03 & 1.3859 & 8.3000e-03 & 1.4705\\ 
  & $40$ & 2.9000e-03 & 1.4458 & 2.9000e-03 & 1.4539 & 2.8000e-03 & 1.4215 & 3.0000e-03 & 1.4681\\ 
  & $80$ & 1.0000e-03 & 1.5361 & 1.0000e-03 & 1.4674 & 1.0000e-03 & 1.4854 & 1.1000e-03 & 1.4475\\  
  & $160$ & 3.7271e-04 & 1.4239 & 3.7584e-04 & 1.4759 & 3.6620e-04 & 1.4493 & 3.7923e-04 & 1.5364\\ 
  & $320$ & 1.3309e-04 & 1.4857 & 1.3465e-04 & 1.4809 & 1.3146e-04 & 1.4780 & 1.3471e-04 & 1.4932\\ 
  \hline
$0.7$& $10$ & 5.1300e-02 &  & 5.1300e-02 &  & 4.9500e-02 &  & 5.3000e-02 & \\ 
  & $20$ & 2.1400e-02 & 1.2613 & 2.1400e-02 & 1.2618 & 2.1000e-02 & 1.2370 & 2.1800e-02 & 1.2817\\ 
  & $40$ & 8.8000e-03 & 1.2820 & 8.8000e-03 & 1.2749 & 8.7000e-03 & 1.2713 & 8.9000e-03 & 1.2925\\ 
  & $80$ & 3.6000e-03 & 1.2895 & 3.6000e-03 & 1.2839 & 3.6000e-03 & 1.2730 & 3.7000e-03 & 1.2663\\  
  & $160$ & 1.5000e-03 & 1.2630 & 1.5000e-03 & 1.2897 & 1.5000e-03 & 1.2630 & 1.5000e-03 & 1.3026\\ 
  & $320$ & 6.0468e-04 & 1.3107 & 6.0624e-04 & 1.2933 & 6.0305e-04 & 1.3146 & 6.0631e-04 & 1.3068\\ 
  \hline
   \end{tabular}
\end{table}
\begin{table}[h!]
\caption{$L_\infty$ errors and corresponding convergence orders for Example \ref{ex2} at $T=1$ for different $\alpha$, $N$ and $\varphi(\tau)$.}
    \label{table:2}
    \centering
    \begin{tabular}{c c c c c c c c c c}
    \hline
\multirow{2}{1em}{$\alpha$} & \multirow{2}{1em}{$N$} &\multicolumn{2}{c} 
   {$\varphi(\tau)=\tau$} &\multicolumn{2}{c} 
   {$\varphi(\tau)=1000(e^{\frac{\tau}{1000}}-1)$} &\multicolumn{2}{c} 
   {$\varphi(\tau)=\sin{\tau}$} &\multicolumn{2}{c} 
   {$\varphi(\tau)=\sinh{\tau}$}\\
   \cline{3-10} 
  & & $E_\infty$ & {rate} & $E_\infty$ & {rate} & $E_\infty$ & {rate} & $E_\infty$ & {rate} \\ \hline
0.3 & $10$ & 5.5000e-03 &  & 5.5000e-03 &  & 3.8000e-03 &  & 7.1000e-03 & \\ 
 &  $20$ & 1.8000e-03 & 1.6114 & 1.8000e-03 & 1.6084 & 1.4000e-03 & 1.4406 & 2.2000e-03 & 1.6903\\ 
 &  $40$ & 5.7302e-04 & 1.6513 & 5.8553e-04 & 1.6251 & 4.6880e-04 & 1.5784 & 6.7725e-04 & 1.6997\\ 
 &  $80$ & 1.8250e-04 & 1.6507 & 1.8875e-04 & 1.6332 & 1.5645e-04 & 1.5833 & 2.0855e-04 & 1.6993\\  
 &  $160$ & 5.7693e-05 & 1.6614 & 6.0818e-05 & 1.6339 & 5.1182e-05 & 1.6120 & 6.4204e-05 & 1.6997\\ 
 &  $320$ & 1.8136e-05 & 1.6695 & 1.9699e-05 & 1.6264 & 1.6508e-05 & 1.6325 & 1.9764e-05 & 1.6998\\ 
   \hline
$0.5$ & $10$ & 1.4900e-02 &  & 1.5000e-02 &  & 1.3200e-02 &  & 1.6600e-02 & \\ 
 & $20$ & 5.5000e-03 & 1.4378 & 5.5000e-03 & 1.4442 & 5.1000e-03 & 1.3720 & 5.9000e-03 & 1.4924\\ 
 &  $40$ & 2.0000e-03 & 1.4594 & 2.0000e-03 & 1.4619 & 1.9000e-03 & 1.4245 & 2.1000e-03 & 1.4903\\ 
 &  $80$ & 7.1285e-04 & 1.4883 & 7.1910e-04 & 1.4726 & 6.8679e-04 & 1.4681 & 7.3891e-04 & 1.5069\\  
 &  $160$ & 2.5489e-04 & 1.4837 & 2.5802e-04 & 1.4787 & 2.4838e-04 & 1.4673 & 2.6140e-04 & 1.4991\\ 
 &  $320$ & 9.0821e-05 & 1.4888 & 9.2384e-05 & 1.4818 & 8.9194e-05 & 1.4775 & 9.2449e-05 & 1.4995\\ 
   \hline
$0.7$ & $10$ & 3.4000e-02 &  & 3.4100e-02 &  & 3.2300e-02 &  & 3.5800e-02 & \\ 
 & $20$ & 1.4200e-02 & 1.2596 & 1.4200e-02 & 1.2654 & 1.3700e-02 & 1.2374 & 1.4600e-02 & 1.2940\\ 
 & $40$ & 5.8000e-03 & 1.2918 & 5.8000e-03 & 1.2790 & 5.7000e-03 & 1.2651 & 5.9000e-03 & 1.3072\\ 
 & $80$ & 2.4000e-03 & 1.2730 & 2.4000e-03 & 1.2871 & 2.4000e-03 & 1.2479 & 2.4000e-03 & 1.2977\\  
 & $160$ & 9.7505e-04 & 1.2995 & 9.7818e-04 & 1.2919 & 9.6853e-04 & 1.3092 & 9.8156e-04 & 1.2899\\ 
 & $320$ & 3.9719e-04 & 1.2956 & 3.9876e-04 & 1.2946 & 3.9557e-04 & 1.2919 & 3.9882e-04 & 1.2993\\ 
   \hline
\end{tabular}
\end{table}
In both examples, the $L_\infty$ errors and their corresponding convergence rates are evaluated for increasing values of $N$ with respect to DFs $\varphi(\tau)=\tau$, $1000(e^{\frac{\tau}{1000}}-1)$, $\sin\tau$, and $\sinh\tau$, for different values of $\alpha$. The results obtained for Examples \ref{ex1} and \ref{ex2} are summarized in Tables \ref{table:1} and \ref{table:2}, respectively. Analysis of the data reveals that all DFs produce satisfactory results. Among the considered DFs, $\varphi(\tau) = \sin \tau$ demonstrates the highest level of accuracy, while $\varphi(\tau) = \sinh \tau$ yields the most rapid error convergence. Furthermore, the NSFD $L_1$ approximation employed in these examples exhibits a convergence order of $\mathcal{O}(\tau^{2-\alpha})$, with both the accuracy and convergence rate improving as $\alpha$ decreases and the number of subintervals increases.
\section{Explicit NSFD Schemes for  \texorpdfstring{1D and 2D}{1D and 2D} TFDEs}\label{sec5}
This portion covers the construction of explicit NSFD schemes for 1D and 2D TFDEs \eqref{diff_eqn} using the $NSL1$ approximation \eqref{nscaputo} for Caputo fractional derivative \eqref{caputo}. To begin with, the domain $\Omega\times[0, T]$ is discretized, in which temporal discretization is performed in the same way as in previous sections. Furthermore, for $d=1$ the space interval $[0, L]$ is divided into $M$ subintervals of steplength $h=\frac{L}{M}$. Hence, the domain ${\Omega}=[0,L]\times[0,T]$ is discretized into a uniform grid ${\omega}=\{(\mathbf{x}_m,t_n):x_m=mh,\ t_n=n\tau;\ m=0,1,\ldots, M,\ n=0,1,\ldots, N;\ M,N\in\mathbb{Z^+}\}$. Similarly, for $d=2$, the spatial domain $[0, L]\times[0,L]$ is disintegrated into $M^2$ number of square subdomains each of sidelength $h=\frac{L}{M}$. Hence, the uniformly discretized 2D space-time grid that is achieved is ${\omega}=\{(\mathbf{x}_m,t_n):x_i=ih,\ y_j=jh,\ t_n=n\tau;\ i,j=0,1,\ldots,M,\ m=j(M+1)+i,\ n=0,1,\ldots, N;\ M,N\in\mathbb{Z^+}\}$. 
\subsection*{\textbf{Explicit NSFD scheme for }$\boldsymbol{1D}$ \textbf{TFDE:}}
Considering TFDE \eqref{diff_eqn} for $d=1$ at grid point $(x_m,t_n)$, we get
\begin{equation*}
_{C}D_{0,t}^{\alpha}u(x_m,t_n)=\partial_x^{2}u(x_m,t_{n-1})+f(x_m,t_{n-1})+\mathcal{O}(\tau).
\end{equation*}
Replacing the temporal and spatial derivatives in the above by their nonstandard approximations \eqref{nscaputo} and \eqref{nsspace}, respectively, becomes
\begin{equation}
\begin{aligned}\label{original}
\frac{\phi(\tau)^{-1}}{\Gamma(2-\alpha)}&\left[b_{1,\alpha}u(x_m,t_n)+\sum_{k=1}^{n-1}(b_{k+1,\alpha}-b_{k,\alpha})u(x_m,t_{n-k})-b_{n,\alpha}u(x_m,t_0)\right]\\
&=\frac{u(x_{m+1},t_{n-1})-2u(x_m,t_{n-1})+u(x_{m-1},t_{n-1})}{\psi(h)}+f(x_m,t_{n-1})+R(x_m,t_n),
\end{aligned}
\end{equation}
where, $R(x_m,t_n)=\mathcal{O}(\tau)+\mathcal{O}(h^2)$. Now approximating $u(x_m,t_n)$ by $u_m^n$ and omitting the truncation error term $R(x_m,t_n)$ in \eqref{original} provides
\begin{equation}\label{duplicate}
\frac{\phi(\tau)^{-1}}{\Gamma(2-\alpha)}\left[b_{1,\alpha}u_m^n+\sum_{k=1}^{n-1}(b_{k+1,\alpha}-b_{k,\alpha})u_m^{n-k}-b_{n,\alpha}u_m^0\right]=\frac{u_{m+1}^{n-1}-2u_m^{n-1}+u_{m-1}^{n-1}}{\psi(h)}+f_m^{n-1}.
\end{equation}
Rearranging the terms in an explicit manner yields
\begin{align}\label{d1scheme}
u_m^n=\mu \left(u_{m+1}^{n-1}-2u_m^{n-1}+u_{m-1}^{n-1}\right)&+\sum_{k=1}^{n-1}\left(b_{k,\alpha}-b_{k+1,\alpha}\right)u_m^{n-k}+b_{n,\alpha}u_m^0+\phi(\tau)\Gamma(2-\alpha)f_m^{n-1},
\end{align}
where, $\mu=\frac{\phi(\tau)\Gamma(2-\alpha)}{\psi(h)}$. The discrete form of initial and Dirichlet boundary conditions \eqref{ini_eqn} and \eqref{bdry_eqn} for $d=1$ can be written as
\begin{align}\label{ini_diff1}
&u_m^0 = u^0(x_m),\ \ 0\leq m\leq M,\\
\label{bdry_diff1}
&\ u_0^n = u_M^n = 0,\ \ 0 \leq n \leq N.
\end{align}
\subsection*{\textbf{Explicit NSFD scheme for }$\boldsymbol{2D}$ \textbf{TFDE:}}
TFDE \eqref{diff_eqn} for $d=2$ when considered at grid point $(x_i,y_j,t_n)$ appears as
\begin{equation*}
_{C}D_{0,t}^{\alpha}u(x_i,y_j,t_n)=\partial_x^{2}u(x_i,y_j,t_{n-1})+\partial_y^{2}u(x_i,y_j,t_{n-1})+f(x_i,y_j,t_{n-1})+\mathcal{O}(\tau).
\end{equation*}
Next, substituting the nonstandard approximations \eqref{nscaputo} and \eqref{nsspace} to the derivatives for $d=2$ and performing a similar set of operations as in the previous case, we obtain 
\begin{equation}
\begin{aligned}\label{original1}
\frac{\phi(\tau)^{-1}}{\Gamma(2-\alpha)}&\left[b_{1,\alpha}u(x_i,y_j,t_n)+\sum_{k=1}^{n-1}(b_{k+1,\alpha}-b_{k,\alpha})u(x_i,y_j,t_{n-k})-b_{n,\alpha}u(x_i,y_j,t_0)\right]\\
&=\frac{u(x_{i+1},y_j,t_{n-1})-2u(x_i,y_j,t_{n-1})+u(x_{i-1},y_j,t_{n-1})}{\psi_1(h)}+f(x_i,y_j,t_{n-1}),\\
&\quad\ +\frac{u(x_i,y_{j+1},t_{n-1})-2u(x_i,y_j,t_{n-1})+u(x_i,y_{j-1},t_{n-1})}{\psi_2(h)}+R(x_i,y_j,t_n).
\end{aligned}
\end{equation}
The truncation error in the above approximation is $R(x_i,y_j,t_n)=\mathcal{O}(\tau)+\mathcal{O}(h^2)$. 
Reorganizing the terms explicitly after removing the truncation error and approximating $u(x_i,y_j,t_n)$ by $u_{i,j}^n$ results in
\begin{equation}
\begin{aligned}\label{d2scheme}
u_{i,j}^n=&\ \mu_1\left(u_{i+1,j}^{n-1}+u_{i-1,j}^{n-1}\right)+\mu_2\left(u_{i,j+1}^{n-1}+u_{i,j-1}^{n-1}\right)-2u_{i,j}^{n-1}(\mu_1+\mu_2)+\sum_{k=1}^{n-1}\left(b_{k,\alpha}-b_{k+1,\alpha}\right)u_{i,j}^{n-k}\\
&+b_{n,\alpha}u_{i,j}^0+\phi(\tau) \Gamma(2-\alpha)f_{i,j}^{n-1},
\end{aligned}
\end{equation}
where, $\mu_1=\frac{\phi(\tau)\Gamma(2-\alpha)}{\psi_1(h)}$ and $\mu_2=\frac{\phi(\tau)\Gamma(2-\alpha)}{\psi_2(h)}$ with $u_{i,j}^n$ being the approximate value to $u(x_i,y_j,t_n)$. In this case, the discrete formulations of 2D initial and Dirichlet boundary conditions \eqref{ini_eqn} and \eqref{bdry_eqn} occurs as
\begin{align}\label{ini_diff2}
&u_{i,j}^0 = u^0(x_i,y_j),\quad 0\leq i,j\leq M,\\
\label{bdry_diff2}
&u_{i,j}^n = 0,\quad \{i=0,M\}\cup\{j=0,M\},\ 0 \leq n \leq N.
\end{align}
\subsection{\textbf{Stability and convergence analysis of explicit NSFD schemes}}
This segment presents the subsequent analysis of the explicit NSFD schemes \eqref{d1scheme}-\eqref{bdry_diff1} and \eqref{d2scheme}-\eqref{bdry_diff2} for the 1D and 2D TFDE problem \eqref{diff_eqn}-\eqref{bdry_eqn}, which comprises the stability and convergence estimates. As mentioned earlier, the prime motive of developing an NSFD scheme is to reduce or eliminate the instability region typically associated with standard finite difference methods. Thus, we are most interested in the stability analysis of the proposed NSFD schemes, which has been performed using the discrete energy method. For further proceedings, the $ L_\infty$ norms of a few vectors to appear in this work are defined as follows:
\[
\begin{aligned}
\text{for }d=1:&
\begin{cases}
\begin{aligned}
    &u^n=(u_1^n,u_2^n,\ldots,u_{M-1}^n)^T, \text{ with } \|u^n\|_\infty=\max_{1\leq m\leq M-1} |u_m^n|;\\
    &f^n=(f_1^n,f_2^n,\ldots,f_{M-1}^n)^T, \text{ with } \|f^n\|_\infty=\max_{1\leq m\leq M-1} |f_m^n|;\quad n=0,1,\ldots,N.
\end{aligned}
\end{cases}
\\\text{for }d=2:&
\begin{cases}
\begin{aligned}
    &u^n=(u_{1,1}^n,u_{2,1}^n,\ldots,u_{M-1,1}^n,u_{1,2}^n,u_{2,2}^n,\ldots,u_{M-1,2}^n,\dots,u_{1,M-1}^n,u_{2,M-1}^n,\ldots,u_{M-1,M-1}^n)^T,\\
    &\hspace{1cm}\text{ with } \|u^n\|_\infty=\max_{1\leq i\leq M-1}\max_{1\leq j\leq M-1} |u_{i,j}^n|;\\
    \\
    &f^n=(f_{1,1}^n,f_{2,1}^n,\ldots,f_{M-1,1}^n,f_{1,2}^n,f_{2,2}^n,\ldots,f_{M-1,2}^n,\dots,f_{1,M-1}^n,f_{2,M-1}^n,\ldots,f_{M-1,M-1}^n)^T,\\
    &\hspace{1cm}\text{ with } \|f^n\|_\infty=\max_{1\leq i\leq M-1}\max_{1\leq j\leq M-1} |f_{i,j}^n|;\quad n=0,1,\ldots,N.
\end{aligned}
\end{cases}
\end{aligned}
\]
\begin{theorem}\label{stab_thm}
Let $\frac{\phi(\tau)}{\psi(h)}\leq\frac{1-2^{-\alpha}}{\Gamma(2-\alpha)}$, then the solution of proposed NSFD scheme \eqref{d1scheme}-\eqref{bdry_diff1} satisfies the following stability estimate, $C$ being a finite positive constant independent of $\tau$ and $h$:
\begin{equation}\label{stab_1}
    \|u^n\|_\infty\leq C\left(\|u^0\|_\infty+\max_{1\leq n\leq N}\|f^{n-1}\|_\infty\right),\quad n=1,2,\ldots,N.
\end{equation}
\end{theorem}
\begin{proof}
The proposed NSFD scheme \eqref{d1scheme} for TFDE \eqref{diff_eqn} $(d=1)$ can be rewritten as
\begin{align}\label{sc1}
&\text{for }n=1:\ u_m^1=\mu \left(u_{m-1}^{0}+u_{m+1}^{0}\right)+(1-2\mu)u_m^{0}+\phi(\tau) \Gamma(2-\alpha) f_m^0,\\\notag
&\text{for } n\geq2:\ u_m^n=\mu \left(u_{m-1}^{n-1}+u_{m+1}^{n-1}\right)+(1-2\mu-b_{2,\alpha})u_m^{n-1}+\sum_{k=2}^{n-1}\left(b_{k,\alpha}-b_{k+1,\alpha}\right)u_m^{n-k}
+b_{n,\alpha}u_m^0\\\label{sc2}
&\hspace{2.7cm}+\phi(\tau) \Gamma(2-\alpha) f_m^{n-1}.
\end{align}
For fix $n\in\{1,2,\dots,N\}$, set $m=m_n$ such that $\left|u_{m_n}^n\right|=\left\|u^n\right\|_\infty$. Now considering \eqref{sc1} at $m=m_0$ and taking absolute values both sides, we get
\begin{align*}
   \left\|u^1\right\|_\infty\leq &\ \mu \left|u_{m_0-1}^{0}+u_{m_0+1}^{0}\right|+\left|1-2\mu\right|\left|u^{0}_{m_0}\right|+\phi(\tau)\Gamma(2-\alpha)\left|f_{m_0}^0\right|,\\
   \leq &\ 2\mu \left\|u^0\right\|_\infty+\left|1-2\mu\right|\left\|u^0\right\|_\infty+\phi(\tau)\Gamma(2-\alpha)\left\|f^0\right\|_\infty.
   \end{align*}
   Since $1-2\mu-b_{2,\alpha}\geq 0$, hence
   \begin{align*}
   \left\|u^1\right\|_\infty\leq & \left\|u^0\right\|_\infty+\frac{\phi(\tau)\Gamma(2-\alpha)}{(1-\alpha)}\left\|f^0\right\|_\infty.
\end{align*}
Let $C=\max\{1,\frac{n^\alpha\phi(\tau)\Gamma(2-\alpha)}{(1-\alpha)}\}$ and $E=\left(\left\|u^0\right\|_\infty+\max_{1\leq n\leq N}\left\|f^{n-1}\right\|_\infty\right)$. 
Also, 
\begin{align}\notag
n^\alpha\phi(\tau)&=n^\alpha\left(\tau^\alpha+\mathcal{O}(\tau^{1+\alpha})\right),\\\notag
&\leq t_n^\alpha+n^{1+\alpha}\mathcal{O}(\tau^{1+\alpha}),\\\notag
&= t_n^\alpha+\mathcal{O}(t_n^{1+\alpha}),\\\label{cbound}
&\leq T^\alpha+\mathcal{O}(T^{1+\alpha})
= \phi(T).
\end{align}
Therefore,
\begin{equation}\label{n1}
  \left\|u^1\right\|_\infty\leq CE.
\end{equation}
Again considering \eqref{sc2} at $m=m_n,\ n\geq2$, and taking absolute values both sides, we get
\begin{align}\notag
   \left\|u^n\right\|_\infty\leq &\ \mu \left|u_{m_n-1}^{n-1}+u_{m_n+1}^{n-1}\right|+\left|1-2\mu-b_{2,\alpha}\right|\left|u_{m_n}^{n-1}\right|+\sum_{k=2}^{n-1}\left(b_{k,\alpha}-b_{k+1,\alpha}\right)\left|u_{m_n}^{n-k}\right|+b_{n,\alpha}\left|u_{m_n}^0\right|\\
   \notag
   &+\phi(\tau) \Gamma(2-\alpha)\left|f_{m_n}^{n-1}\right|,
   \end{align}
   \begin{align}
   \leq &\ 2\mu \left\|u^{n-1}\right\|_\infty+(1-2\mu-b_{2,\alpha})\left\|u^{n-1}\right\|_\infty+\sum_{k=2}^{n-1}\left(b_{k,\alpha}-b_{k+1,\alpha}\right)\left\|u^{n-k}\right\|_\infty+b_{n,\alpha}\left\|u^{0}\right\|_\infty\\\notag
   &+\phi(\tau) \Gamma(2-\alpha)\left\|f^{n-1}\right\|_\infty,\\\notag
   \leq &\ (1-b_{2,\alpha})\left\|u^{n-1}\right\|_\infty+\sum_{k=2}^{n-1}\left(b_{k,\alpha}-b_{k+1,\alpha}\right)\left\|u^{n-k}\right\|_\infty+b_{n,\alpha}\left\|u^{0}\right\|_\infty+\phi(\tau) \Gamma(2-\alpha)\left\|f^{n-1}\right\|_\infty,\\
   \label{tobeproved}
   \leq &\ b_{n,\alpha}CE+(1-b_{2,\alpha})\left\|u^{n-1}\right\|_\infty+\sum_{k=2}^{n-1}\left(b_{k,\alpha}-b_{k+1,\alpha}\right)\left\|u^{n-k}\right\|_\infty.
\end{align}
At this point, we claim that 
\begin{equation}\label{claim}
      \left\|u^n\right\|_\infty\leq CE,\quad n=1,2,\dots,N.
\end{equation}
The case $n=1$ for relation \eqref{claim} follows straightforwardly from \eqref{n1}. Hence, we apply the process of mathematical induction by considering this as the base case. To prove further cases of $n=2,\ldots, N$, we assume \eqref{claim} to be true for all $n=2,\ldots,j-1$. Thus substituting $n=j$ in \eqref{tobeproved}, it is observed that
\begin{align*}
    \left\|u^j\right\|_\infty\leq &\ b_{j,\alpha}CE+(1-b_{2,\alpha})\left\|u^{j-1}\right\|_\infty+\sum_{k=2}^{j-1}\left(b_{k,\alpha}-b_{k+1,\alpha}\right)\left\|u^{j-k}\right\|_\infty,\\
    \leq &\ CE\left(1-b_{2,\alpha}+\sum_{k=2}^{j-1}\left(b_{k,\alpha}-b_{k+1,\alpha}\right)+b_{j,\alpha}\right),\\
    \leq &\ CE.
\end{align*}
Therefore, $\left\|u^n\right\|_\infty\leq CE,\ n=1,2,\dots,N$.
\end{proof}
\begin{remark}
    In the above analysis, the stability region $\frac{\phi(\tau)}{\psi(h)}\leq \frac{1-2^{-\alpha}}{\Gamma(2-\alpha)}$ is surely broader than the stability region $\frac{\tau^\alpha}{h^2}\leq\frac{1-2^{-\alpha}}{\Gamma(2-\alpha)}$ of the SFD scheme \cite{li2015numerical} with any $\phi(\tau)< \tau^\alpha$ and  $\psi(h)> h^2$.\par Particularly, when $\tau=h$ and $\phi(\tau)^\frac{1}{\alpha}=\sqrt{\psi(h)}=e^h-1$ (say), then 
    \begin{align*}     
    \text{for SFD scheme, } &h\geq\left(\frac{\Gamma(2-\alpha)}{1-2^{-\alpha}}\right)^\frac{1}{2-\alpha},\\
    \text{for NSFD scheme, } &h\geq\ln{\left(1+\left(\frac{\Gamma(2-\alpha)}{1-2^{-\alpha}}\right)^\frac{1}{2-\alpha}\right)}.
    \end{align*}
    Clearly, the NSFD scheme has a larger stability region than the SFD scheme and is going to give stable solutions for comparatively smaller $h$.
\end{remark}
\begin{theorem}\label{stab_thm1}
Let $\left(\frac{\phi(\tau)}{\psi_1(h)}+\frac{\phi(\tau)}{\psi_2(h)}\right)\leq\frac{1-2^{-\alpha}}{\Gamma(2-\alpha)}$, then the solution of proposed NSFD scheme \eqref{d2scheme}-\eqref{bdry_diff2} satisfies the following stability estimate, $C$ being a finite positive constant independent of $\tau$ and $h$:
\begin{equation}\label{stab_2}
    \|u^n\|_\infty\leq C\left(\|u^0\|_\infty+\max_{1\leq n\leq N}\|f^{n-1}\|_\infty\right),\quad n=1,2,\ldots,N.
\end{equation}
\end{theorem}
\begin{proof}
The proposed NSFD scheme \eqref{d2scheme} for TFDE \eqref{diff_eqn} $(d=2)$ can be rewritten as
\begin{align}\label{dsc1}
&\text{for }n=1:\ u_{i,j}^1=\mu_1\left(u_{i+1,j}^0+u_{i-1,j}^0\right)+\mu_2\left(u_{i,j+1}^0+u_{i,j-1}^0\right)+u_{i,j}^0(1-2\mu_1-2\mu_2)+\phi(\tau)\Gamma(2-\alpha)f_{i,j}^{n-1},\\\notag
&\text{for } n\geq2:\ u_{i,j}^n=\mu_1\left(u_{i+1,j}^{n-1}+u_{i-1,j}^{n-1}\right)+\mu_2\left(u_{i,j+1}^{n-1}+u_{i,j-1}^{n-1}\right)+u_{i,j}^{n-1}(1-b_{2,\alpha}-2\mu_1-2\mu_2)+b_{n,\alpha}u_{i,j}^0\\\label{dsc2}
&\hspace{3cm}+\sum_{k=2}^{n-1}\left(b_{k,\alpha}-b_{k+1,\alpha}\right)u_{i,j}^{n-k}+\phi(\tau) \Gamma(2-\alpha)f_{i,j}^{n-1}.
\end{align}
For fix $n\in\{1,2,\dots,N\}$, set $i=i_n$ and $j=j_n$ such that $\left|u_{i_n,j_n}^n\right|=\left\|u^n\right\|_\infty$. Now, evaluating  \eqref{dsc1} at $(i_0,j_0,1)$ and taking the absolute value on both sides, we obtain
\begin{align*}
   \left\|u^1\right\|_\infty\leq &\ \mu_1 \left|u_{i_0-1,j_0}^{0}+u_{i_0+1,j_0}^{0}\right|+\mu_2 \left|u_{i_0,j_0-1}^{0}+u_{i_0,j_0+1}^{0}\right|+\left|1-2\mu_1+2\mu_2\right|\left|u^{0}_{i_0,j_0}\right|+\phi(\tau)\Gamma(2-\alpha)\left|f_{i_0,j_0}^0\right|,\\
   \leq &\ (2\mu_1+2\mu_2) \left\|u^0\right\|_\infty+\left|1-2\mu_1-2\mu_2\right|\left\|u^0\right\|_\infty+\phi(\tau)\Gamma(2-\alpha)\left\|f^0\right\|_\infty.
   \end{align*}
   which implies
   \begin{align*}
   \left\|u^1\right\|_\infty\leq & \left\|u^0\right\|_\infty+\frac{\phi(\tau)\Gamma(2-\alpha)}{(1-\alpha)}\left\|f^0\right\|_\infty.
\end{align*}
Again let $C=\max\{1,\frac{n^\alpha\phi(\tau)\Gamma(2-\alpha)}{(1-\alpha)}\}$ and $E=\left(\left\|u^0\right\|_\infty+\max_{1\leq n\leq N}\left\|f^{n-1}\right\|_\infty\right)$. Using \eqref{cbound}, this implies,
\begin{equation}\label{n11}
  \left\|u^1\right\|_\infty\leq CE.
\end{equation}
Again for $n\geq2$, considering \eqref{dsc2} at $(i_n,j_n,n)$, and taking absolute values both sides, it becomes
\begin{align}\label{n21}
   \left\|u^n\right\|_\infty
   \leq &\ b_{n,\alpha}CE+(1-b_{2,\alpha})\left\|u^{n-1}\right\|_\infty+\sum_{k=2}^{n-1}\left(b_{k,\alpha}-b_{k+1,\alpha}\right)\left\|u^{n-k}\right\|_\infty.
\end{align}
Now, since the inequalities \eqref{n11} and \eqref{n21} for the proposed solution are the same as in the previous proof, therefore, 
\begin{equation*}
    \left\|u^n\right\|_\infty\leq CE,\ n=1,2,\dots, N.
\end{equation*}
\end{proof}
The following theorems establish the order of convergence of proposed NSFD schemes \eqref{d1scheme}-\eqref{bdry_diff1} and \eqref{d2scheme}-\eqref{bdry_diff2} for the 1D and 2D TFDE IBVPs defined by \eqref{diff_eqn}-\eqref{bdry_eqn}, utilizing the stability estimates \eqref{stab_1} and \eqref{stab_2}, respectively. Beforehand, we define the error terms, the associated error vectors, and their discrete \(L_\infty\) norms. Let the pointwise errors at the grid node \((x_m, t_n)\) and \((x_i,y_j, t_n)\) be denoted by
$e_m^n = u(x_m, t_n) - u_m^n$ and $e_{i,j}^n = u(x_i,y_j, t_n) - u_{i,j}^n$, respectively. Then for a fixed time level \(n \in \{1, 2, \dots, N\}\), we define
\begin{align*}
\text{for }d=1: &\begin{cases}
     e^n = \left(e_1^n, e_2^n, \dots, e_{M-1}^n\right)^T, \text{ with }\ \|e^n\|_\infty = \max_{1 \leq m \leq M-1} |e_m^n|.
     \end{cases}
     \\
\text{for }d=2:&\begin{cases}
    e^n=(e_{1,1}^n,e_{2,1}^n,\ldots,e_{M-1,1}^n,e_{1,2}^n,e_{2,2}^n,\ldots,e_{M-1,2}^n,\dots,e_{1,M-1}^n,e_{2,M-1}^n,\ldots,e_{M-1,M-1}^n)^T,\\
    \text{ with } \|e^n\|_\infty=\max_{1\leq i\leq M-1}\max_{1\leq j\leq M-1} |e_{i,j}^n|.
    \end{cases}
\end{align*}
\begin{theorem}
If $\frac{\phi(\tau)}{\psi(h)}\leq\frac{1-2^{-\alpha}}{\Gamma(2-\alpha)}$, then the error $e_{m}^{n}$ of the proposed NSFD scheme \eqref{d1scheme}-\eqref{bdry_diff1}, for small enough $h$ and $\tau$,  satisfy the following error estimate:
\begin{equation}\label{err_est}
    \|e^n\|_\infty=\mathcal{O}(\tau^{2-\alpha}+h^2),\ \ n=1,2,\ldots,N.
\end{equation}
\end{theorem}
\begin{proof}
By subtracting the numerical scheme \eqref{duplicate} from the discrete scheme \eqref{original} satisfied by the exact solution, we obtain the error equation as
\begin{align*}
    \frac{\phi(\tau)^{-1}}{\Gamma(2-\alpha)}\bigg[b_{1,\alpha}e_m^n+&\sum_{k=1}^{n-1}(b_{k+1,\alpha}-b_{k,\alpha})e_m^{n-k}-b_{n,\alpha}e_m^0\bigg]=\frac{e_{m+1}^{n-1}-2e_m^{n-1}+e_{m-1}^{n-1}}{\psi(h^2)}+\theta_m^n,\\
    &m=1,2,\ldots,M-1 \text{ and } n=1,2,\ldots,N;\\
    &\hspace{.8cm}e_m^0=0,\ m=0,1,2,\ldots,M;\\
    &\hspace{.5cm}e_0^n=e_M^n=0,\ n=0,1,2,\ldots,N.
\end{align*}
where $\theta_m^n=\mathcal{O}(\tau+h^2)$ is truncation error at $(x_m,t_n)$. Since the above difference scheme for error $e_m^n$ has the same structure as $u_m^n$ in \eqref{duplicate}. Hence, from viewpoint of Theorem \ref{stab_thm}, estimate \eqref{stab_1} will hold for $e_m^n$ also and thus we obtain
\begin{align*}
    \|e^n\|_\infty&\leq C \max_{1\leq n\leq N}\left(\|\theta^n\|_\infty\right),\\
    &\leq C^\prime (\tau+h^2).
\end{align*}
Therefore, we get $\left\|e^n\right\|_\infty=\mathcal{O}(\tau+h^2),\ \forall n=1,2,\dots,N.$
\end{proof}
\begin{theorem}
If $\left(\frac{\phi(\tau)}{\psi_1(h)}+\frac{\phi(\tau)}{\psi_2(h)}\right)\leq\frac{1-2^{-\alpha}}{\Gamma(2-\alpha)}$, then the error $e_{i,j}^{n}$ of the proposed NSFD scheme \eqref{d2scheme}-\eqref{bdry_diff2}, for sufficiently small $h$ and $\tau$, satisfies the following error estimate:
\begin{equation}\label{err_est2}
    \|e^n\|_\infty = \mathcal{O}(\tau^{2-\alpha} + h^2), \quad n = 1,2,\ldots,N.
\end{equation}
\end{theorem}
\begin{proof}
By subtracting the numerical scheme \eqref{d2scheme} from the discrete scheme \eqref{original1} satisfied by the exact solution, we obtain the error equation as
\begin{align*}
    \frac{\phi(\tau)^{-1}}{\Gamma(2-\alpha)}\bigg[ b_{1,\alpha}e_{i,j}^n + \sum_{k=1}^{n-1}(b_{k+1,\alpha} - &b_{k,\alpha})e_{i,j}^{n-k} - b_{n,\alpha}e_{i,j}^0 \bigg] =  \frac{e_{i+1,j}^{n-1} - 2e_{i,j}^{n-1} + e_{i-1,j}^{n-1}}{\psi(h)} + \frac{e_{i,j+1}^{n-1} - 2e_{i,j}^{n-1} + e_{i,j-1}^{n-1}}{\psi(h)}\\
    &\hspace{3.6cm}+ \theta_{i,j}^n,\quad i, j = 1, 2, \ldots, M-1,\ n = 1, 2, \ldots, N,\\
&e_{i,j}^0 = 0,\quad 0\leq i,j\leq M,\\
&e_{i,j}^n = 0,\quad \{i=0,M\}\cup\{j=0,M\},\ 0 \leq n \leq N.
\end{align*}
Here, $\theta_{i,j}^n = \mathcal{O}(\tau + h^2)$ denotes the local truncation error at the grid point $(x_i, y_j, t_n)$. Note that the above error equation has the same structure as \eqref{d2scheme} for $u_{i,j}^n$. Therefore, by applying the stability result from Theorem \ref{stab_thm1}, the following bound holds:
\begin{align*}
    \|e^n\|_\infty &\leq C \max_{1 \leq k \leq n} \|\theta^k\|_\infty, \\
                  &\leq C^\prime (\tau + h^2),
\end{align*}
for all $n = 1,2,\ldots,N$. Thus,
\[
\|e^n\|_\infty = \mathcal{O}(\tau + h^2).
\]
This completes the proof.
\end{proof}

\section{Numerical Experiments and Results for Explicit NSFD Schemes}\label{sec6}
In this section, the theoretical findings of the previous section are validated through a series of numerical experiments conducted on some test examples. Specifically, two test examples are considered in which the first involves a 1D IBVP described by TFDE \eqref{diff_eqn}-\eqref{bdry_eqn}, while the second extends the validation to a 2D setting of the same problem. All numerical experiments are carried out at a fixed fractional order $\alpha = 0.9$, under varying conditions like increasing numbers of spatial subintervals and different choices of DFs. For each test example, the exact analytical solution is known beforehand, which enables precise computation of numerical errors and facilitates a comparative analysis across different configurations. The assessment of numerical schemes' performance is carried out using various tools, such as error norms, convergence order computations, error distribution plots, and solution surface plots. The norm used for computing error at each time level is the $L_\infty$ (maximum absolute) norm. The formulas for computing error and convergence rate at final time $T$ are given as:
\begin{itemize}
  \item \textbf{For the one-dimensional case ($d = 1$):}
\begin{equation*}
    E_\infty(N, M) = \max_{0 \leq m \leq M} \left| u(x_m, t^N) - u_m^N \right|, \quad 
    S_{rate} = \log_2 \left( \frac{E_\infty(N, M/2)}{E_\infty(N, M)} \right).
\end{equation*}
\item \textbf{For the two-dimensional case ($d = 2$):}
  \begin{equation*}
    E_\infty(N, M) = \max_{0 \leq i,j \leq M} \left| u(x_i, y_j, t^N) - u_{i,j}^N \right|, \quad 
    S_{rate} = \log_2 \left( \frac{E_\infty(N, M/2)}{E_\infty(N, M)} \right).
  \end{equation*}
\end{itemize}
Let us now consider the first test example to verify the NSFD scheme \eqref{d1scheme}-\eqref{bdry_diff1} for 1D TFDE IBVP \eqref{diff_eqn}-\eqref{bdry_eqn}.
\begin{example}\label{ex3}
Consider the 1D TFDE problem \eqref{diff_eqn}-\eqref{bdry_eqn} with following data:
\begin{align*}
    &u^0(x)=0,\ x\in[0,1],\\
    &f(x,t)=t^3\sin\pi x\left(\frac{\Gamma\left(4+\alpha\right)}{3!}+\pi^2 t^\alpha\right),\quad t\in[0,1].
\end{align*}
Here, the domain is $\Omega=[0,1]\times[0,1]$. It is readily verified that the exact solution of the problem is given by $u(x,t) = t^{3 + \alpha} \sin(\pi x)$.
This solution satisfies the zero initial condition and explicitly depends on the fractional order \(\alpha\).
\end{example}
In this example, the aim is to examine the performance of the proposed NSFD scheme \eqref{d1scheme} in terms of accuracy and stability under various choices of DFs. Also, since the spatial error convergence dominates the temporal one, hence, the primary focus is on evaluating the spatial convergence of the scheme. With this in mind, numerical experiments are carried out for a fixed fractional order \(\alpha=0.9 \) and $N=10000$, while systematically increasing the number of spatial subintervals. The following spatial DFs are considered in this analysis:
\[
\psi(h)=4\sin^2\left(\frac{h}{2}\right),\ \sin^2(h),\ \left(100\left(1 - e^{-h/100}\right)\right)^2,\ \frac{4}{\pi^2}\sinh^2\left(\frac{\pi h}{2}\right),\ \sinh^2(h),\ \left(100\left(e^{h/100} - 1\right)\right)^2,\ \sinh(h^2).
\]\par
\begin{table}[ht]
\caption{Maximum absolute errors and convergence rate at $T=1$, $\alpha=0.9$, and $N=10000$ with $\phi(\tau)=\tau^\alpha$.}
    \label{table:3}
    \centering
    \begin{tabular}{c c c c c c c c c}
    \hline
  \multirow{2}{1em}{$M$} &\multicolumn{2}{c}{$\psi(h)=h^2$} &\multicolumn{2}{c}{$\psi(h)=4\sin^2\frac{h}{2}$} &\multicolumn{2}{c}{$\psi(h)=\sin^2h$} &\multicolumn{2}{c}{$\psi(h)=(100(1-e^\frac{-h}{100}))^2$}\\
\cline{2-9} 
   & $E_\infty$ & $T_{rate}$ & $E_\infty$ & $T_{rate}$ & $E_\infty$ & $T_{rate}$ & $E_\infty$ & $T_{rate}$\\ \hline
   $2$ & 1.6000e-01 &  & 1.4350e-01 &  & 9.4400e-02 & & 1.5600e-01 &\\ 
   $2^2$ & 3.8100e-02 & 2.0702 & 3.4200e-02 & 2.0690 & 2.2600e-02 & 2.0625 & 3.6200e-02 & 2.1075\\ 
   $2^3$ & 9.4000e-03 & 2.0191 & 8.4000e-03 & 2.0255 & 5.5000e-03 & 2.0388 & 8.4000e-03 & 2.1075\\ 
   $2^4$ & 2.3000e-03 & 2.0310 & 2.0000e-03 & 2.0704 & 1.3000e-03 & 2.0809 & 1.8000e-03 & 2.2224\\  
   $2^5$ & 5.1666e-04 & 2.1543 & 4.5721e-04 & 2.1291 & 2.7888e-04 & 2.2208 & 2.8840e-04 & 2.6419\\
   \hline
    \end{tabular}
\end{table}
\begin{table}[ht]
\caption{Maximum absolute errors and convergence rate at $T=1$, $\alpha=0.9$, and $N=10000$ with $\phi(\tau)=\tau^\alpha$.}
    \label{table:4}
    \centering
    \begin{tabular}{c c c c c c c c c}
    \hline
  \multirow{2}{1em}{$M$} &\multicolumn{2}{c}{$\psi(h)=\frac{4}{\pi^2}\sinh^2(\frac{\pi h}{2})$} &\multicolumn{2}{c}{$\psi(h)=\sinh^2h$} &\multicolumn{2}{c}{$\psi(h)=(100(e^\frac{h}{100}-1))^2$} &\multicolumn{2}{c}{$\psi(h)=\sinh{h^2}$}\\
\cline{2-9} 
   & $E_\infty$ & $T_{rate}$ & $E_\infty$ & $T_{rate}$ & $E_\infty$ & $T_{rate}$ & $E_\infty$ & $T_{rate}$\\ \hline
   $2$   & 3.2460e-01 &        & 2.2640e-01 &        & 1.6400e-01 &        & 1.6820e-01 &        \\ 
   $2^2$ & 7.6700e-02 & 2.0814 & 5.3700e-02 & 2.0759 & 4.0000e-02 & 2.0356 & 3.8600e-02 & 2.1235 \\ 
   $2^3$ & 1.8800e-02 & 2.0285 & 1.3200e-02 & 2.0244 & 1.0300e-02 & 1.9574 & 9.4000e-03 & 2.0379 \\ 
   $2^4$ & 4.6000e-03 & 2.0310 & 3.2000e-03 & 2.0444 & 2.7000e-03 & 1.9316 & 2.3000e-03 & 2.0310 \\ 
   $2^5$ & 1.1000e-03 & 2.0641 & 7.5446e-04 & 2.0846 & 7.4496e-04 & 1.8577 & 5.1678e-04 & 2.1540 \\ 
   \hline
    \end{tabular}
\end{table}
\begin{table}[!h]
\caption{Maximum absolute errors and convergence rate at $T=1$, $\alpha=0.9$, and $N=10000$ with $\phi(\tau)=\left(\frac{1-e^{-100\tau}}{100}\right)^\alpha$.}
    \label{table:5}
    \centering
    \begin{tabular}{c c c c c c c c c}
    \hline
  \multirow{2}{1em}{$M$} &\multicolumn{2}{c}{$\psi(h)=\frac{4}{\pi^2}\sinh^2(\frac{\pi h}{2})$} &\multicolumn{2}{c}{$\psi(h)=\sinh^2h$} &\multicolumn{2}{c}{$\psi(h)=(100(e^\frac{h}{100}-1))^2$} &\multicolumn{2}{c}{$\psi(h)=\sinh{h^2}$}\\
\cline{2-9} 
   & $E_\infty$ & $T_{rate}$ & $E_\infty$ & $T_{rate}$ & $E_\infty$ & $T_{rate}$ & $E_\infty$ & $T_{rate}$\\ \hline
   $2$   & 3.2250e-01 &        & 2.2450e-01 &        & 1.6230e-01 &        & 1.6660e-01 &        \\ 
   $2^2$ & 7.5200e-02 & 2.1005 & 5.2300e-02 & 2.1018 & 3.8600e-02 & 2.0720 & 3.7300e-02 & 2.1591 \\ 
   $2^3$ & 1.7500e-02 & 2.1034 & 1.1900e-02 & 2.1358 & 9.0000e-03 & 2.1006 & 8.1000e-03 & 2.2032 \\ 
   $2^4$ & 3.4000e-03 & 2.3637 & 2.0000e-03 & 2.5729 & 1.5000e-03 & 2.5850 & 1.1000e-03 & 2.8804 \\ 
   $2^5$ & 7.1358e-04 & 2.2524 & 4.6168e-04 & 2.1150 & 3.7115e-04 & 2.0149 & 1.9876e-04 & 2.4684 \\ 
   \hline
\end{tabular}
\end{table}
Tables \ref{table:3} and \ref{table:4} present numerical results obtained using various nonstandard DFs, while the temporal DF is fixed as standard \(\phi(\tau) = \tau^\alpha\). The results in Table \ref{table:3} correspond to those spatial DFs that yield improved accuracy over the standard spatial DF \(\psi(h) = h^2\). Notably, the choice \(\psi(h) = \sin^2(h)\) produces the most accurate solutions, whereas \(\psi(h) = \left(100\left(1 - e^{-h/100}\right)\right)^2\) demonstrates enhanced error convergence behavior. Table \ref{table:4} presents results for the remaining spatial DFs, which do not outperform Table \ref{table:3}, but still exhibit reasonably accurate results and good convergence. Furthermore, the spatial DFs in Table \ref{table:4} are re-evaluated in combination with a nonstandard temporal DF \(\phi(\tau) = \left(\frac{1 - e^{-100\tau}}{100}\right)^\alpha\). As shown in Table \ref{table:5}, this pairing yields better results in both accuracy and convergence of the numerical solutions.\par
\begin{figure}[htbp]
	\begin{center}
		\centering
	\subfigure[]{%
		\includegraphics[scale=0.375]{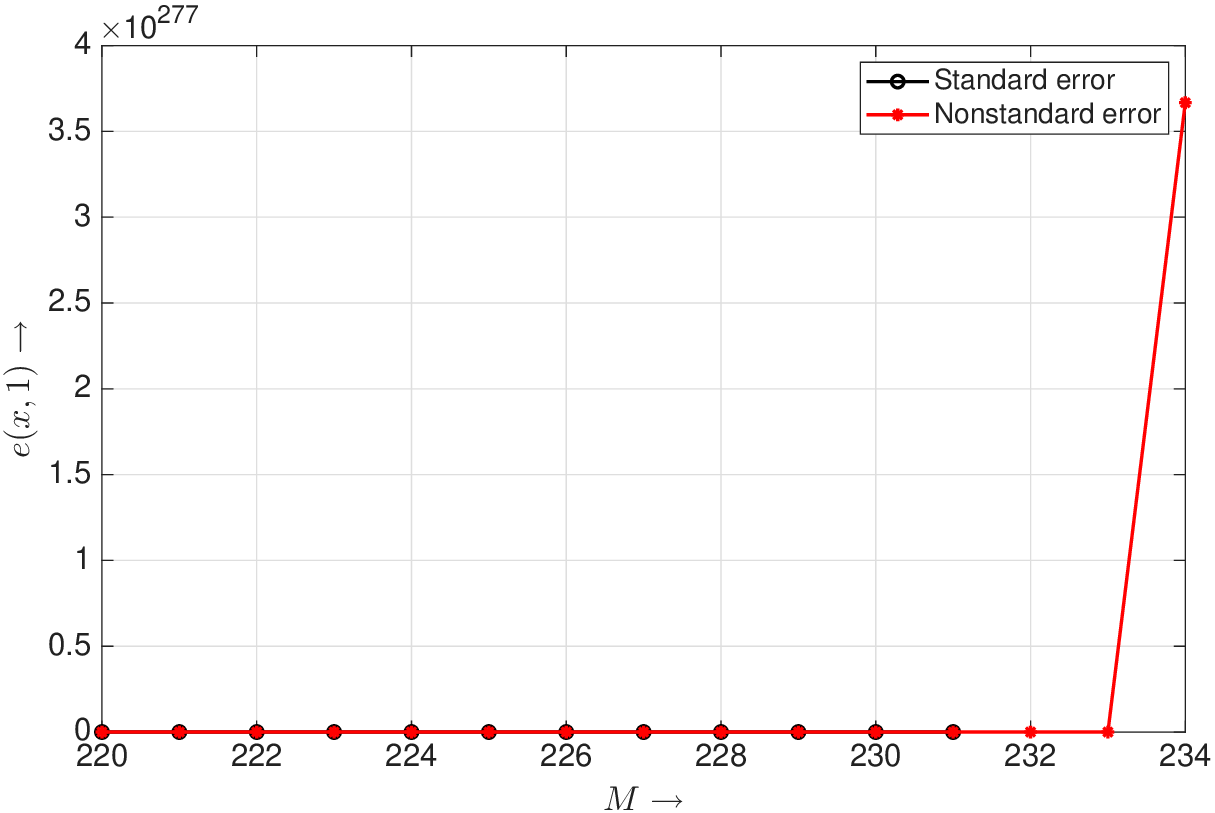}
			\label{subfigure8}}
   \subfigure[]{%
		\includegraphics[scale=0.375]{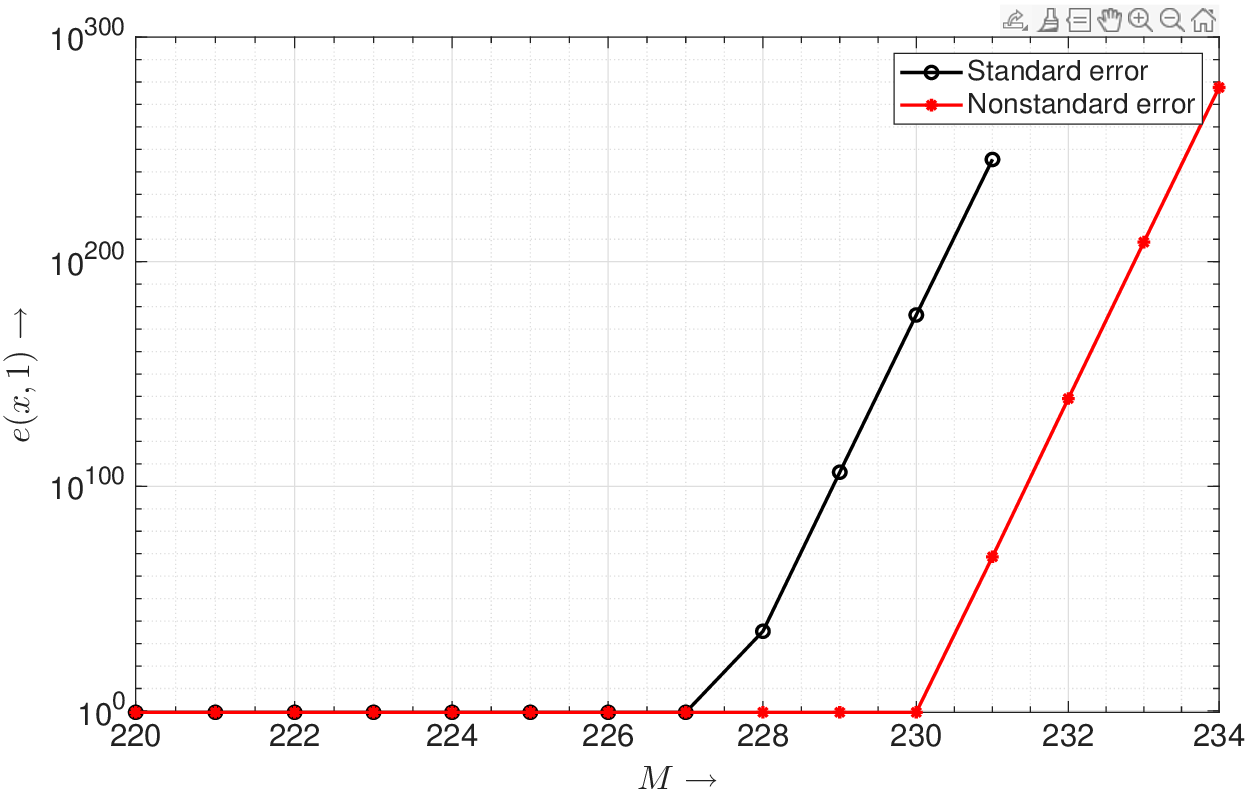}
			\label{subfigure9}} 
       \caption{The (a) linear and (b) Log-log plots for the maximum absolute errors of standard and NSFD approximations at final time $T=1$ and $\alpha=0.9$ against the number of spatial subintervals} 
   \label{fig3}
	\end{center}
\end{figure}
\begin{figure}[ht]
	\begin{center}
		\centering
	\subfigure[$\psi(h)=h^2$]{%
		\includegraphics[scale=0.379]{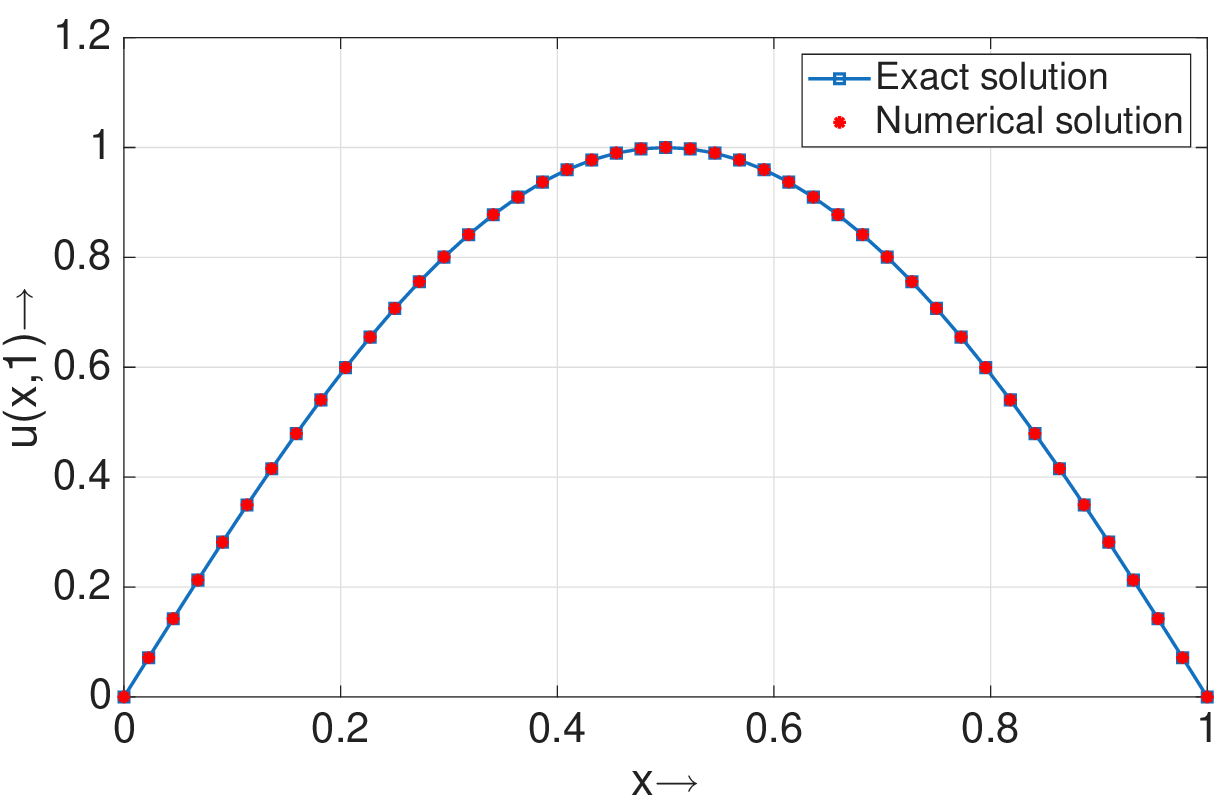}
			\label{subfigure10}}
   \subfigure[$\psi(h)=\sin^2h$]{%
		\includegraphics[scale=0.379]{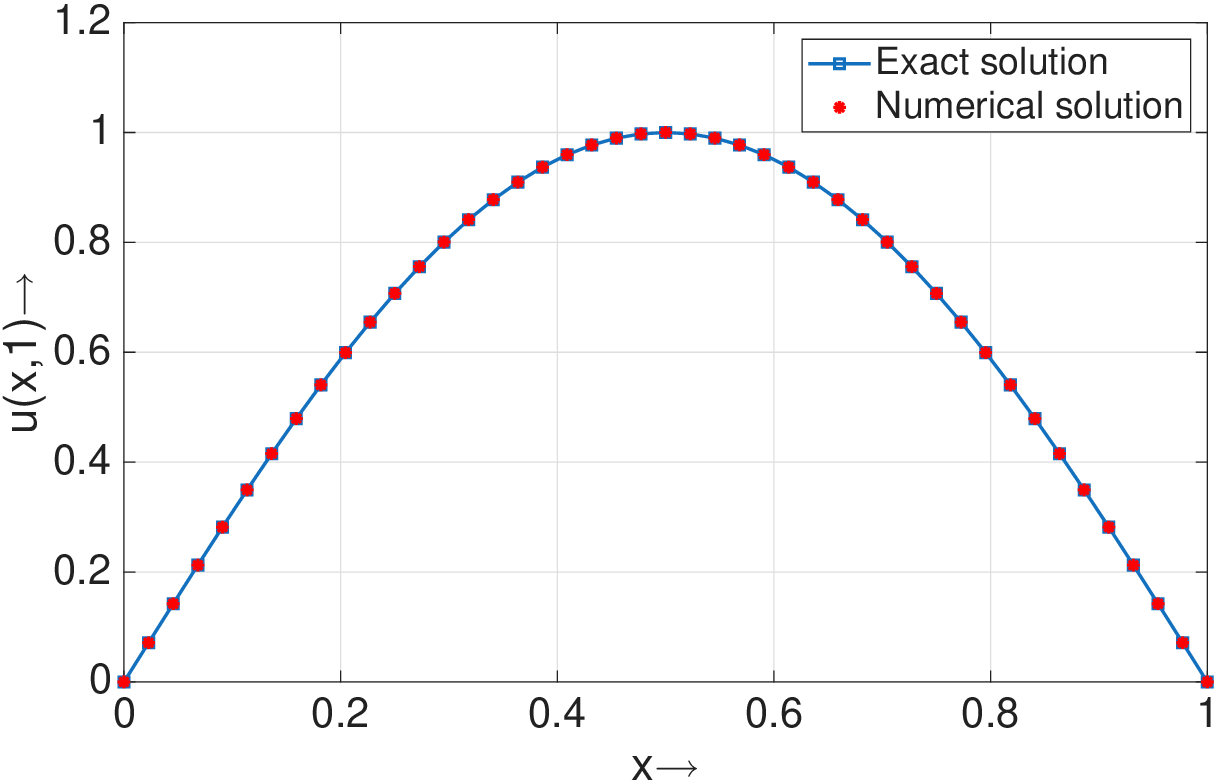}
			\label{subfigure11}} 
            \\
            \subfigure[$\psi(h)=4\sin^2\frac{h}{2}$]{%
		\includegraphics[scale=0.379]{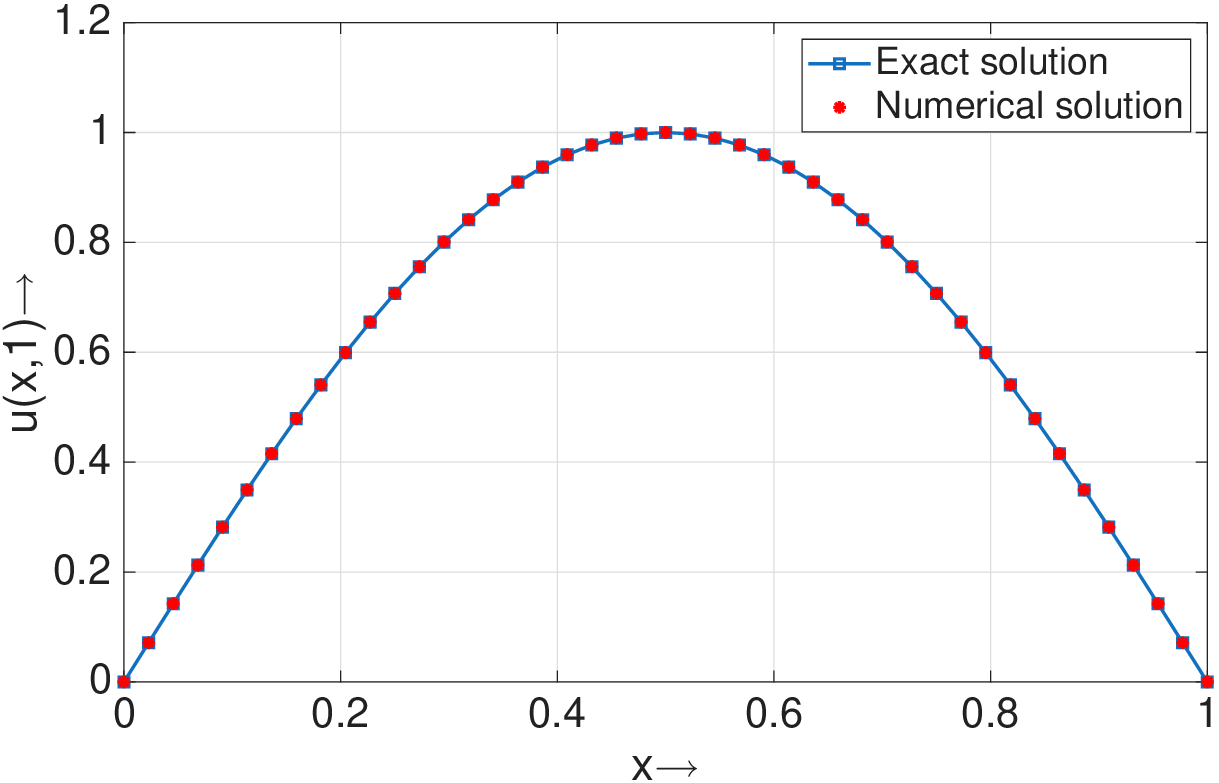}
			\label{subfigure100}}
   \subfigure[$\psi(h)=(100(1-e^\frac{-h}{100}))^2$]{%
		\includegraphics[scale=0.379]{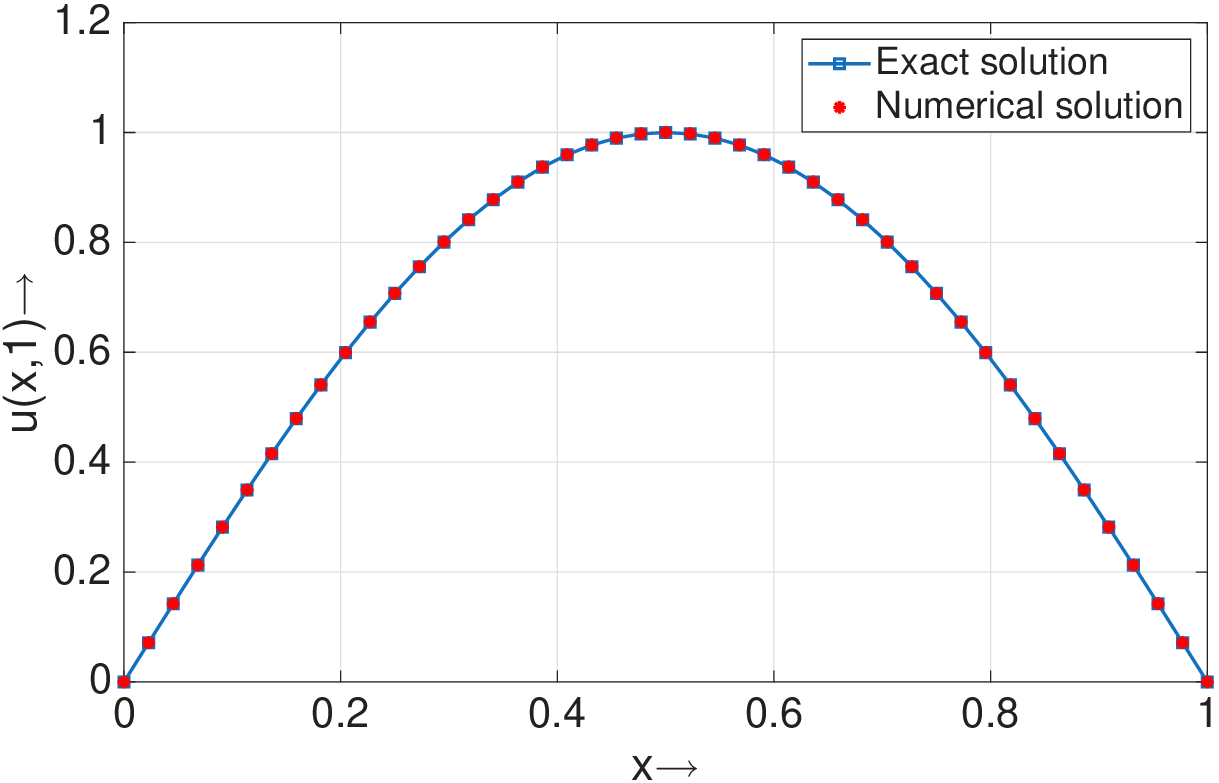}
			\label{subfigure101}} 
       \caption{Comparison of exact and numerical solutions of Example \ref{ex3} at final time $T=1$ with $\alpha=0.9$ and $M=44,\ N=10000$ for SFD and NSFD approximations.} 
   \label{fig4}
	\end{center}
\end{figure}
In an alternate setting, we choose \(L = 5.0848\), $\alpha=0.9$, $N=10000$, and $\phi(\tau)=\tau^\alpha$, such that the proposed scheme \eqref{d1scheme} with standard DF $\psi(h)=h^2$, remains stable up to \(M = 227\) steps. Whereas, the nonstandard DF $\psi(h)=(e^h-1)^2$ extends the stability beyond \(M = 227\).
This behavior is represented in Figure \ref{fig3}, which clearly illustrates the improved stability achieved through the NSFD scheme. Consequently, this confirms the effectiveness of the NSFD scheme in enlarging the stability region relative to its standard counterpart. The curve plots for comparison of the exact analytic solution to the numerical solution using SFD and different NSFD approximations at final time $T=1$ and $\alpha=0.9$ are given in Figure \ref{fig4}.\par
Next example is considered to verify the NSFD scheme \eqref{d1scheme}-\eqref{bdry_diff1} for solving 2D TFDE \eqref{diff_eqn}-\eqref{bdry_eqn}.
\begin{example}\label{ex4}
Consider the 2D TFDE problem \eqref{diff_eqn}-\eqref{bdry_eqn} with following data:
\begin{align*}
    &u^0(x,y)=0,\ x\in[0,1],\ y\in[0,1],\\
    &f(x,y,t)=t^3\sin\pi x\ \sin\pi y\left(\frac{\Gamma\left(4+\alpha\right)}{3!}+2\pi^2 t^\alpha\right),\quad t\in[0,1].
\end{align*}
Here, the domain is $\Omega=[0,1]\times[0,1]\times[0,1]$. It is clear from the above data that the exact solution is given by $u(x,t) = t^{3 + \alpha} \sin(\pi x)\sin(\pi y)$.
\end{example}
Here also, we focus on evaluating the important characteristics such as accuracy, efficiency, and stability of the proposed NSFD scheme \eqref{d2scheme}. These properties are verified with respect to different choices of DFs. Again, the spatial errors are calculated since the spatial convergence dominates the temporal one, and hence, the spatial convergence of the scheme is observed. To this end, a few numerical experiments are conducted considering different test cases by increasing the number of subintervals and changing spatial DFs, but with a fixed fractional order \(\alpha=0.9 \) and $N=10000$. The spatial DFs mentioned in the previous example are also considered in this analysis.
\begin{table}[ht]
\caption{Maximum absolute errors and convergence rate at $T=1$, $\alpha=0.9$, and $N=20000$ with $\phi(\tau)=\tau^\alpha$.}
\label{table:6}
\centering
\begin{tabular}{c c c c c c c c c}
\hline
\multirow{2}{1em}{$M$} & \multicolumn{2}{c}{$\psi(h)=h^2$} & \multicolumn{2}{c}{$\psi(h)=4\sin^2\frac{h}{2}$} & \multicolumn{2}{c}{$\psi(h)=\sin^2h$} & \multicolumn{2}{c}{$\psi(h)=(100(1-e^{-h/100}))^2$} \\
\cline{2-9}
& $E_\infty$ & $T_{rate}$ & $E_\infty$ & $T_{rate}$ & $E_\infty$ & $T_{rate}$ & $E_\infty$ & $T_{rate}$ \\
\hline
$2$   & 1.9100e-01 &        & 1.7080e-01 &        & 1.1140e-01 &        & 1.8610e-01 &        \\
$2^2$ & 4.4500e-02 & 2.1017 & 4.0000e-02 & 2.0942 & 2.6300e-02 & 2.0826 & 4.2300e-02 & 2.1373 \\
$2^3$ & 1.0900e-02 & 2.0295 & 9.8000e-03 & 2.0291 & 6.5000e-03 & 2.0166 & 9.9000e-03 & 2.0952 \\
$2^4$ & 2.7000e-03 & 2.0133 & 2.4000e-03 & 2.0297 & 1.6000e-03 & 2.0224 & 2.2000e-03 & 2.1699 \\
$2^5$ & 6.6126e-04 & 2.0297 & 5.9227e-04 & 2.0187 & 3.8532e-04 & 2.0539 & 3.9637e-04 & 2.4726 \\
\hline
\end{tabular}
\end{table}
\begin{table}[ht]
\caption{Maximum absolute errors and convergence rate at $T=1$, $\alpha=0.9$, and $N=20000$ with $\phi(\tau)=\tau^\alpha$.}
\label{table:7}
\centering
\begin{tabular}{c c c c c c c c c}
\hline
\multirow{2}{1em}{$M$} 
& \multicolumn{2}{c}{$\psi(h)=\frac{4}{\pi^2}\sinh^2(\frac{\pi h}{2})$} 
& \multicolumn{2}{c}{$\psi(h)=\sinh^2h$} 
& \multicolumn{2}{c}{$\psi(h)=(100(e^\frac{h}{100}-1))^2$} 
& \multicolumn{2}{c}{$\psi(h)=\sinh{h^2}$} \\
\cline{2-9} 
& $E_\infty$ & $T_{rate}$ & $E_\infty$ & $T_{rate}$ & $E_\infty$ & $T_{rate}$ & $E_\infty$ & $T_{rate}$\\ 
\hline
$2$   & 3.9920e-01 &        & 2.7340e-01 &        & 1.9580e-01 &        & 2.0110e-01 &        \\ 
$2^2$ & 9.0200e-02 & 2.1459 & 6.2900e-02 & 2.1199 & 4.6700e-02 & 2.0679 & 4.5100e-02 & 2.1567 \\ 
$2^3$ & 2.2000e-02 & 2.0356 & 1.5400e-02 & 2.0301 & 1.2000e-02 & 1.9604 & 1.1000e-02 & 2.0356 \\ 
$2^4$ & 5.4000e-03 & 2.0265 & 3.8000e-03 & 2.0189 & 3.2000e-03 & 1.9069 & 2.7000e-03 & 2.0265 \\ 
$2^5$ & 1.3000e-03 & 2.0544 & 9.3725e-04 & 2.0195 & 9.2623e-04 & 1.7886 & 6.6140e-04 & 2.0294 \\ 
\hline
\end{tabular}
\end{table}
\begin{table}[!h]
\caption{Maximum absolute errors and convergence rate at $T=1$, $\alpha=0.9$, and $N=10000$ with $\phi(\tau)=\left(\frac{1-e^{-100\tau}}{100}\right)^\alpha$.}
\label{table:8}
\centering
\begin{tabular}{c c c c c c c c c}
\hline
\multirow{2}{1em}{$M$} 
& \multicolumn{2}{c}{$\psi(h)=\frac{4}{\pi^2}\sinh^2(\frac{\pi h}{2})$} 
& \multicolumn{2}{c}{$\psi(h)=\sinh^2h$} 
& \multicolumn{2}{c}{$\psi(h)=(100(e^\frac{h}{100}-1))^2$} 
& \multicolumn{2}{c}{$\psi(h)=\sinh{h^2}$} \\
\cline{2-9} 
& $E_\infty$ & $T_{rate}$ & $E_\infty$ & $T_{rate}$ & $E_\infty$ & $T_{rate}$ & $E_\infty$ & $T_{rate}$\\ 
\hline
$2$   & 3.9850e-01 &        & 2.7280e-01 &        & 1.9530e-01 &        & 2.0060e-01 &        \\ 
$2^2$ & 8.9800e-02 & 2.1498 & 6.2500e-02 & 2.1259 & 4.6400e-02 & 2.0735 & 4.4700e-02 & 2.1660 \\ 
$2^3$ & 2.1600e-02 & 2.0557 & 1.5000e-02 & 2.0589 & 1.1600e-02 & 2.0000 & 1.0600e-02 & 2.0762 \\ 
$2^4$ & 5.1000e-03 & 2.0825 & 3.5000e-03 & 2.0995 & 2.9000e-03 & 2.0000 & 2.4000e-03 & 2.1430 \\ 
$2^5$ & 9.9770e-04 & 2.3538 & 5.9292e-04 & 2.5614 & 5.8190e-04 & 2.3172 & 3.1726e-04 & 2.9193 \\ 
\hline
\end{tabular}
\end{table}
In this example, Tables \ref{table:6} and \ref{table:7} report numerical findings corresponding to different choices of nonstandard spatial DFs, with the temporal DF fixed as the standard form $\phi(\tau) = \tau^\alpha$. Table \ref{table:6} includes those spatial DFs that lead to improved accuracy relative to the classical choice $\psi(h) = h^2$. Among them, $\psi(h) = \sin^2(h)$ yields the most accurate solutions, while $\psi(h) = \left(100\left(1 - e^{-h/100}\right)\right)^2$ shows notable improvements in error convergence. Table \ref{table:7} compiles those results for the remaining spatial DFs, which although do not outperform those in Table \ref{table:6} but still deliver reliable accuracy and good convergence rates. These spatial DFs are further reconsidered under a modified temporal DF, $\phi(\tau) = \left(\frac{1-e^{-100\tau}}{100}\right)^\alpha$, as presented in Table \ref{table:8}. The new temporal DF leads to visibly improved performance in terms of both accuracy and convergence behavior. The contour plots for comparison of the exact analytic solution to the numerical solutions using SFD and different NSFD approximations at final time $T=1$ and $\alpha=0.9$ are given in Figure \ref{22}.
\begin{figure}[!ht]
    \begin{center}
	\centering
		\subfigure[exact solution]{
	\includegraphics[width=7.714cm]{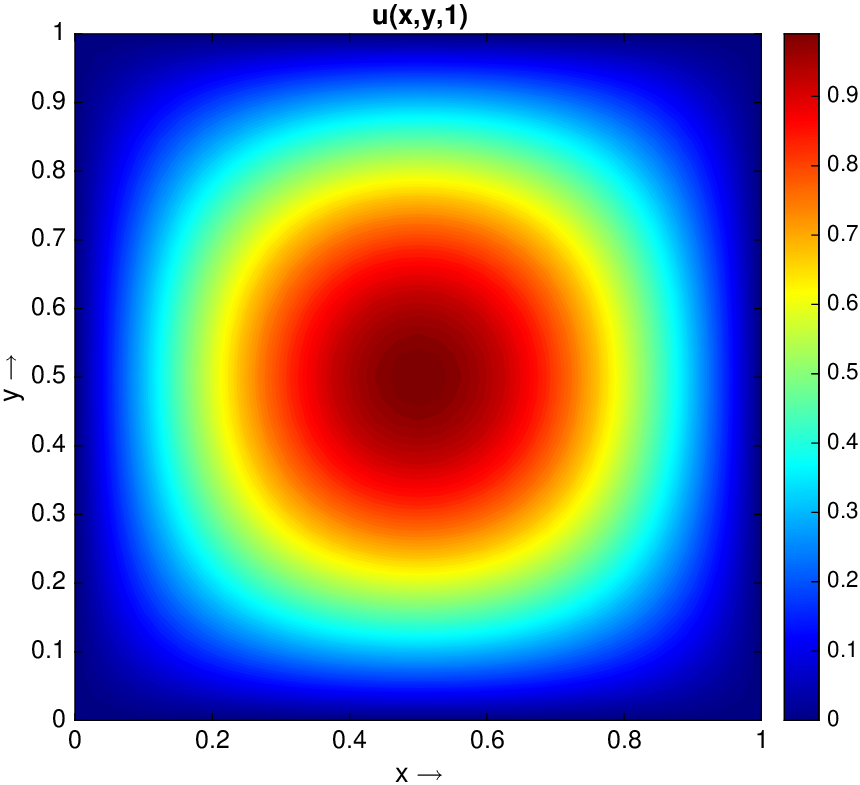}
			\label{subfigure17}} 
		\subfigure[$\psi(h)=h^2$]{
	\includegraphics[width=7.714cm]{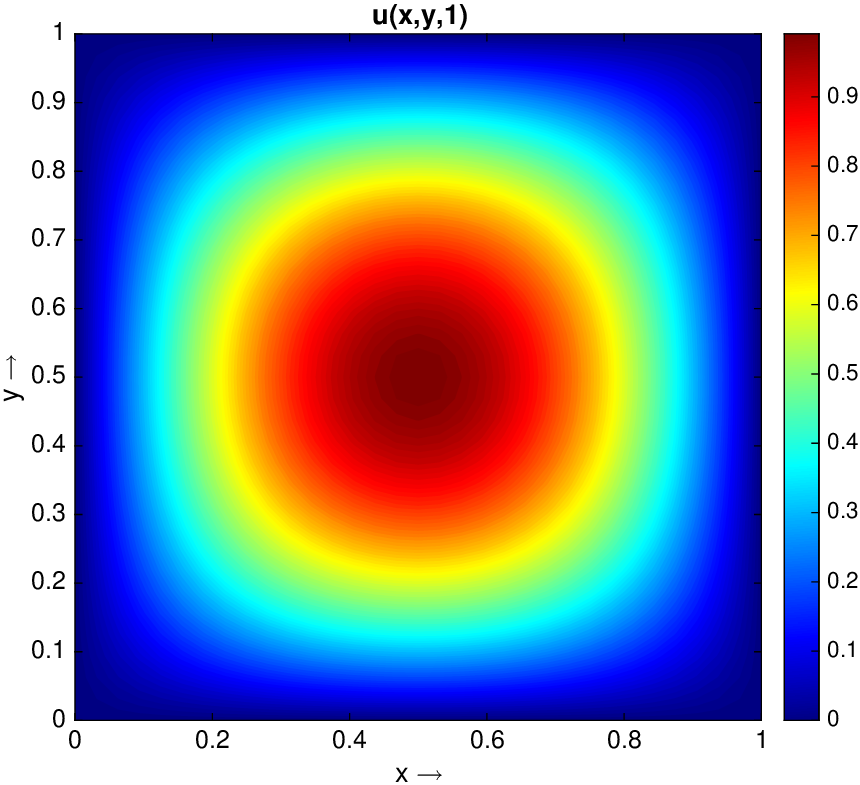}			 \label{subfigure21}} 
    \\
	\centering
		\subfigure[$\psi(h)=\sin^2h$]{
	\includegraphics[width=7.714cm]{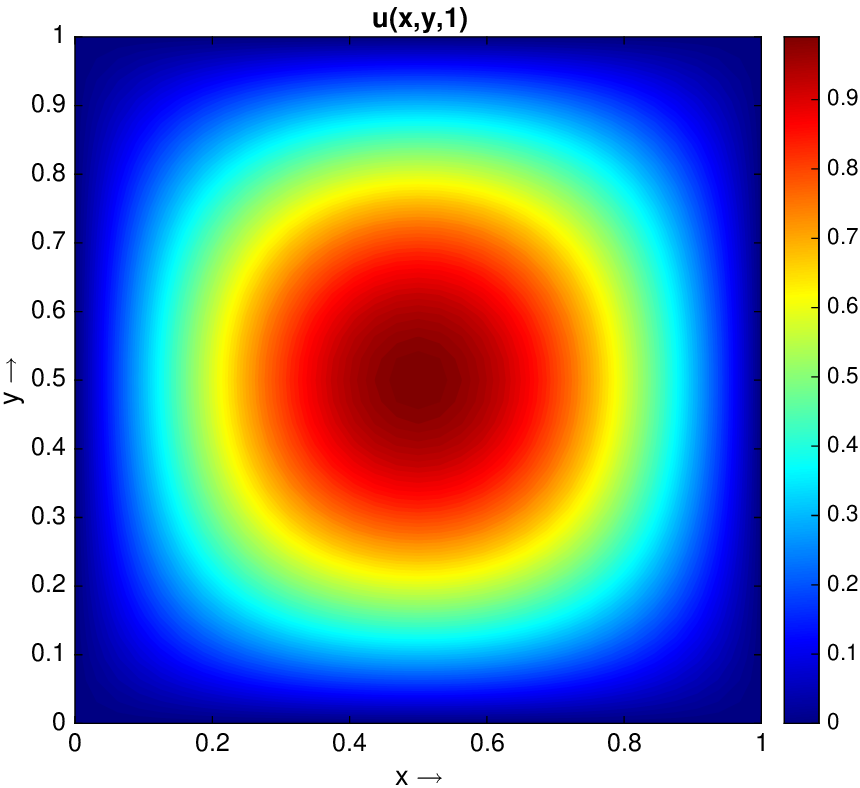}
			\label{subfigure18}} 
		\subfigure[$\psi(h)=(100(1-e^{-h/100}))^2$]{
	\includegraphics[width=7.714cm]{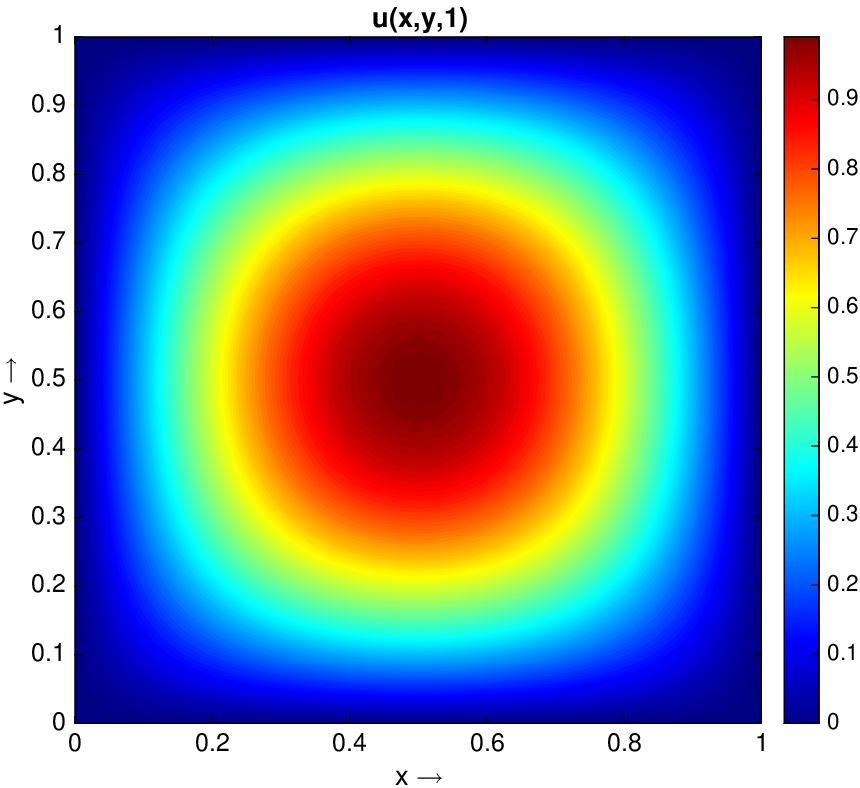}			 \label{subfigure22}} 
    \caption{Contour plots of exact and numerical solutions of Example \ref{ex4} at final time $T=1$ with $\alpha=0.9$ and $M=32,\ N=20000$ for SFD and NSFD approximations.} 
    \label{22}
    \end{center}
\end{figure}
\section{Conclusion}\label{sec7}
This study confirms the applicability and superiority of explicit NSFD methods over SFD schemes for solving one and two-dimensional Caputo-type TFDEs. A nonstandard L1 approximation to the Caputo fractional derivative is first introduced. This formulation is \(A(\alpha)\)-stable and possesses a truncation error of order \(\mathcal{O}(\tau^{2 - \alpha})\). A set of numerical experiments validates this proposed formulation and the associated numerical error. This approximation is then employed to construct explicit NSFD schemes using carefully chosen DFs. The resulting schemes are shown to be stable under the conditions  
\[
\frac{\phi(\tau)}{\psi(h)} \leq \frac{1 - 2^{-\alpha}}{\Gamma(2 - \alpha)} \quad \text{(1D)}, \quad 
\frac{\phi(\tau)}{\psi_1(h)} + \frac{\phi(\tau)}{\psi_2(h)} \leq \frac{1 - 2^{-\alpha}}{\Gamma(2 - \alpha)} \quad \text{(2D)},
\]
where \( \alpha \) is the fractional order. For appropriate choices of DFs, these conditions yield stability regions significantly broader than those of classical SFD schemes. The proposed methods are also shown to be convergent with order \( \mathcal{O}(\tau + h^2) \). A comprehensive set of numerical experiments corroborates the theoretical results, demonstrating that the proposed NSFD schemes offer superior accuracy, stability, and efficiency in the explicit discretization of Caputo-type TFDEs. The graphical representations of the numerical results present the idea of NSFD schemes more clearly. The explicit nature of these methods renders them computationally attractive for large-scale simulations and parallel implementations.
Future work will focus on extending the proposed framework to nonlinear TFDEs, problems with variable coefficients, and systems posed on complex geometries or networks such as metric graphs.

\section*{Statements and Declarations}
\subsection*{\textbf{Funding and acknowledgements}}
The first author acknowledges the support provided by the Council of Scientific and Industrial Research (CSIR), India, for research in this paper under grant number 09/086(1483)/2020-EMR-I.

\subsection*{\textbf{Competing interests}}
The authors declare that there are no financial and non-financial competing interests that are relevant to the content of this article.

\subsection*{\textbf{Data availability}}
Data sharing is not applicable to this article as no datasets were generated or analysed during the current study.

\bibliographystyle{elsarticle-num}
\bibliography{refer}

\begin{thebibliography}{10}
\expandafter\ifx\csname url\endcsname\relax
  \def\url#1{\texttt{#1}}\fi
\expandafter\ifx\csname urlprefix\endcsname\relax\def\urlprefix{URL }\fi
\expandafter\ifx\csname href\endcsname\relax
  \def\href#1#2{#2} \def\path#1{#1}\fi

\bibitem{podlubny1998fractional}
I.~Podlubny, Fractional differential equations: an introduction to fractional
  derivatives, fractional differential equations, to methods of their solution
  and some of their applications, Vol. 198, elsevier, 1998.

\bibitem{samko1993fractional}
S.~G. Samko, Fractional integrals and derivatives, Theory and applications
  (1993).

\bibitem{oldham1974fractional}
K.~Oldham, J.~Spanier, The fractional calculus theory and applications of
  differentiation and integration to arbitrary order, Vol. 111, Elsevier, 1974.

\bibitem{kilbas2006theory}
A.~A. Kilbas, H.~M. Srivastava, J.~J. Trujillo, Theory and applications of
  fractional differential equations, Vol. 204, elsevier, 2006.

\bibitem{diethelm2010analysis}
K.~Diethelm, N.~Ford, The analysis of fractional differential equations,
  Lecture notes in mathematics 2004 (2010).

\bibitem{kumari2024numerical}
S.~Kumari, M.~Mehra, Numerical solution to loaded difference scheme for
  time-fractional diffusion equation with temporal loads, Journal of
  Mathematical Chemistry (2024) 1--27.

\bibitem{magin2010fractional}
R.~L. Magin, Fractional calculus models of complex dynamics in biological
  tissues, Computers \& Mathematics with Applications 59~(5) (2010) 1586--1593.

\bibitem{mehandiratta2020difference}
V.~Mehandiratta, M.~Mehra, A difference scheme for the time-fractional
  diffusion equation on a metric star graph, Applied Numerical Mathematics 158
  (2020) 152--163.

\bibitem{mehandiratta2023well}
V.~Mehandiratta, M.~Mehra, G.~Leugering, Well-posedness, optimal control and
  discretization for time-fractional parabolic equations with time-dependent
  coefficients on metric graphs, Asian Journal of Control 25~(3) (2023)
  2360--2377.

\bibitem{shah2019new}
R.~Shah, H.~Khan, U.~Farooq, D.~Baleanu, P.~Kumam, M.~Arif, A new analytical
  technique to solve system of fractional-order partial differential equations,
  IEEE Access 7 (2019) 150037--150050.

\bibitem{engheta2002fractional}
N.~Engheta, On fractional calculus and fractional multipoles in
  electromagnetism, IEEE Transactions on Antennas and Propagation 44~(4) (2002)
  554--566.

\bibitem{kumari2025finite}
S.~Kumari, M.~Mehra, V.~Mehandiratta, Finite difference approximation of
  time-fractional advection-diffusion equation on a metric star graph,
  International Journal of Computer Mathematics (2025) 1--19.

\bibitem{carreras2001anomalous}
B.~Carreras, V.~Lynch, G.~Zaslavsky, Anomalous diffusion and exit time
  distribution of particle tracers in plasma turbulence model, Physics of
  Plasmas 8~(12) (2001) 5096--5103.

\bibitem{koeller1984applications}
R.~Koeller, Applications of fractional calculus to the theory of
  viscoelasticity (1984).

\bibitem{singh2025non}
A.~K. Singh, M.~Mehra, R.~Pulch, Non-local physics informed neural networks for
  forward and inverse problems containing non-local operators, Neural Computing
  and Applications 37~(6) (2025) 4111--4132.

\bibitem{jin2023numerical}
B.~Jin, Z.~Zhou, Numerical treatment and analysis of time-fractional evolution
  equations, Vol. 214, Springer, 2023.

\bibitem{kumari2024high}
S.~Kumari, A.~K. Singh, V.~Mehandiratta, M.~Mehra, High-order approximation to
  {C}aputo derivative on graded mesh and time-fractional diffusion equation for
  nonsmooth solutions, Journal of Computational and Nonlinear Dynamics 19~(10)
  (2024) 101001.

\bibitem{kanwal2024explicit}
A.~Kanwal, S.~Boulaaras, R.~Shafqat, B.~Taufeeq, M.~ur~Rahman, Explicit scheme
  for solving variable-order time-fractional initial boundary value problems,
  Scientific Reports 14~(1) (2024) 5396.

\bibitem{chen2013superlinearly}
M.~Chen, W.~Deng, Y.~Wu, Superlinearly convergent algorithms for the
  two-dimensional space--time {C}aputo--{R}iesz fractional diffusion equation,
  Applied Numerical Mathematics 70 (2013) 22--41.

\bibitem{mickens1993nonstandard}
R.~E. Mickens, Nonstandard finite difference models of differential equations,
  world scientific, 1993.

\bibitem{mickens2002nonstandard}
R.~E. Mickens, Nonstandard finite difference schemes for differential
  equations, Journal of Difference Equations and Applications 8~(9) (2002)
  823--847.

\bibitem{mickens2004positivity}
R.~E. Mickens, P.~Jordan, A positivity-preserving nonstandard finite difference
  scheme for the damped wave equation, Numerical Methods for Partial
  Differential Equations: An International Journal 20~(5) (2004) 639--649.

\bibitem{kayenat2024some}
S.~Kayenat, A.~Kumar~Verma, Some novel exact and non-standard finite difference
  schemes for a class of diffusion--advection--reaction equation, Journal of
  Difference Equations and Applications 30~(11) (2024) 1747--1764.

\bibitem{mickens2020note}
R.~E. Mickens, T.~M. Washington, A note on a positivity preserving nonstandard
  finite difference scheme for a modified parabolic
  reaction--advection--diffusion {PDE}, Journal of Difference Equations and
  Applications 26~(11-12) (2020) 1423--1427.

\bibitem{moaddy2011non}
K.~Moaddy, S.~Momani, I.~Hashim, The non-standard finite difference scheme for
  linear fractional {PDE}s in fluid mechanics, Computers \& Mathematics with
  Applications 61~(4) (2011) 1209--1216.

\bibitem{agarwal2018non}
P.~Agarwal, A.~El-Sayed, Non-standard finite difference and {C}hebyshev
  collocation methods for solving fractional diffusion equation, Physica A:
  Statistical Mechanics and Its Applications 500 (2018) 40--49.

\bibitem{taghipour2022efficient}
M.~Taghipour, H.~Aminikhah, An efficient non-standard finite difference scheme
  for solving distributed order time fractional reaction--diffusion equation,
  International Journal of Applied and Computational Mathematics 8~(2) (2022)
  56.

\bibitem{dwivedi2019numerical}
K.~D. Dwivedi, S.~Das, Numerical solution of the nonlinear diffusion equation
  by using non-standard/standard finite difference and {F}ibonacci collocation
  methods, The European Physical Journal Plus 134~(12) (2019) 608.

\bibitem{anileyresults}
W.~T. Aniley, G.~F. Duressa, Results in applied mathematics.

\bibitem{li2015numerical}
C.~Li, F.~Zeng, Numerical methods for fractional calculus, CRC Press, 2015.

\bibitem{qiao2022fast}
H.~Qiao, A.~Cheng, A fast high order method for time fractional diffusion
  equation with non-smooth data, Discrete and Continuous Dynamical Systems-B
  27~(2) (2022) 903--920.

\bibitem{lin2007finite}
Y.~Lin, C.~Xu, Finite difference/spectral approximations for the
  time-fractional diffusion equation, Journal of computational physics 225~(2)
  (2007) 1533--1552.

\bibitem{mickens2000applications}
R.~E. Mickens, Applications of nonstandard finite difference schemes, World
  Scientific, 2000.

\bibitem{mickens2005advances}
R.~E. Mickens, Advances in the applications of nonstandard finite difference
  schemes, World Scientific, 2005.

\bibitem{mickens2007calculation}
R.~E. Mickens, Calculation of denominator functions for nonstandard finite
  difference schemes for differential equations satisfying a positivity
  condition, Numerical Methods for Partial Differential Equations: An
  International Journal 23~(3) (2007) 672--691.

\bibitem{elaydiintroduction}
S.~N. Elaydi, et~al., An introduction to difference equations [electronic
  resource].

\end{thebibliography}

\end{document}